\documentclass[final,onefignum,onetabnum]{siamart250211}

\usepackage{amsmath,amsfonts,amsopn,amssymb}
\usepackage{graphicx}
\usepackage[caption=false]{subfig}
\usepackage{multirow,relsize,makecell}
\usepackage{url}
\ifpdf
\DeclareGraphicsExtensions{.eps,.pdf,.png,.jpg}
\else
\DeclareGraphicsExtensions{.eps}
\fi
\usepackage{algorithm, algorithmic}
\usepackage[shortlabels]{enumitem}
\usepackage{accents}

\numberwithin{theorem}{section}
\newsiamremark{assumption}{Assumption}
\newsiamremark{remark}{Remark}
\newsiamremark{question}{Question}
\newsiamremark{example}{Example}
\crefname{assumption}{Assumption}{Assumptions}
\crefname{remark}{Remark}{Remarks}
\crefname{example}{Example}{Examples}

\headers{Auxiliary space theory}{Jongho Park and Jinchao Xu}

\title{Auxiliary space theory for the analysis of\\ iterative methods for semidefinite linear systems\thanks{Submitted to arXiv.
\funding{This work was supported by the KAUST Baseline Research Fund.}
}}

\author{
Jongho Park\thanks{Applied Mathematics and Computational Sciences Program, Computer, Electrical and Mathematical Science and Engineering Division, King Abdullah University of Science and Technology~(KAUST), Thuwal 23955, Saudi Arabia
 (\email{jongho.park@kaust.edu.sa}, \email{jinchao.xu@kaust.edu.sa}).}
 \and
Jinchao Xu\footnotemark[2]
}

\ifpdf
\hypersetup{
  pdftitle={Auxiliary space theory for the analysis of iterative methods for semidefinite linear systems},
  pdfauthor={Jongho Park and Jinchao Xu}
}
\fi

\usepackage[normalem]{ulem}
\usepackage{color}

\begin{document}

\maketitle
\begin{abstract}
We present an auxiliary space theory that provides a unified framework for analyzing various iterative methods for solving linear systems that may be semidefinite. 
By interpreting a given iterative method for the original system as an equivalent, yet more elementary, iterative method for an auxiliary system defined on a larger space, we derive sharp convergence estimates using elementary linear algebra.  
In particular, we establish identities for the error propagation operator and the condition number associated with iterative methods, which generalize and refine existing results.  
The proposed auxiliary space theory is applicable to the analysis of numerous advanced numerical methods in scientific computing.  
To illustrate its utility, we present three examples---subspace correction methods, Hiptmair--Xu preconditioners, and auxiliary grid methods---and demonstrate how the proposed theory yields refined analyses for these cases.
\end{abstract}



\section{Introduction}
\label{Sec:Introduction}
This paper concerns iterative methods for solving the following general linear system defined on a finite-dimensional vector space~$V$:
\begin{equation}
\label{model}
A u = f,
\end{equation}
where $A \colon V \to V$ is a symmetric positive semidefinite~(semi-SPD) linear operator, and $f$ belongs to the range of $A$.
We denote by $u$ a solution of~\eqref{model}.
More precisely, we propose an auxiliary space theory, based on the idea of the auxiliary space method~\cite{Xu:1996}, which provides a unified framework for analyzing a broad class of iterative methods for solving linear systems of the form~\eqref{model}.

The auxiliary space method, proposed in~\cite{Xu:1996}, is a framework for designing iterative methods for linear systems based on an auxiliary space.
This approach has led to the development of several notable numerical methods. 
For example, in~\cite{GWX:2016,Xu:1996,ZX:2014}, efficient multigrid methods for problems defined on unstructured grids were developed by using more structured auxiliary spaces.
In~\cite{HX:2007,HX:2008}, preconditioners for problems in $H(\operatorname{curl})$ and $H(\operatorname{div})$, now known as the Hiptmair--Xu preconditioners, were developed based on the auxiliary space method.
In these preconditioners, a vector equation in three dimensions is reduced to solving four scalar Poisson equations, which are significantly easier to handle.
Other relevant works include~\cite{KV:2009,KV:2012}.
There have also been efforts to extend the idea of the Hiptmair--Xu preconditioners to broader settings, such as the de Rham complex and exterior calculus~\cite{BBH:2020,GNV:2018}.
Additional applications of the auxiliary space method can be found in~\cite{CHH:2018,CWWY:2015,HXZ:2013,LX:2016}.

Beyond its role in designing efficient iterative methods, the auxiliary space method has also been utilized from a theoretical perspective.
In~\cite{Chen:2011}, subspace correction methods~\cite{Xu:1992}, which provide a unified framework for various iterative methods such as domain decomposition and multigrid methods (see, e.g.,~\cite{TW:2005}), were analyzed within the framework of the auxiliary space method.
In particular, the Xu--Zikatanov identity~\cite{XZ:2002}, which gives a sharp estimate for the convergence rate of the successive subspace correction method~(see also~\cite{Brenner:2013,CXZ:2008}), was elegantly proven using the auxiliary space method.
In~\cite{XZ:2017}, the auxiliary space method was used to provide a unified understanding of various algebraic multigrid methods.

In this paper, we focus on the theoretical aspect of the auxiliary space method.
We consider the general iterative method
\begin{equation}
\label{iterative}
u^{m+1} = u^{m} + B(f - Au^{m}),
\quad m \geq 0,
\end{equation}
where $B \colon V \to V$ is a linear operator.
Specifically, we introduce an auxiliary space theory that generalizes existing frameworks~\cite{Nepomnyaschikh:1992,Xu:1992,XZ:2017} and naturally accommodates semidefinite problems.
While the analysis of iterative methods for solving semidefinite linear systems has been extensively studied in the literature (e.g.,~\cite{Cao:2008,LWXZ:2006,LWXZ:2008,WLXZ:2008}), often relying on specialized techniques to address singularities, our approach addresses such cases in a uniform and elementary manner, without invoking advanced mathematical tools.
Moreover, we provide sharp estimates for the error propagation operator and the condition number associated with iterative methods, in the sense that we derive exact identities for them (cf.~\cite{Brenner:2013,XZ:2002}).
An important corollary of our results is a sharp convergence analysis for general iterative methods, as well as for the preconditioned conjugate gradient method applied to semidefinite problems.
We note that a sharp theory for general iterative methods in symmetric positive definite~(SPD) settings was developed in~\cite{XZ:2017}, and the conjugate gradient method for singular problems has been studied in~\cite{Hayami:2018,Kaasschieter:1988}.

As applications of the auxiliary space theory developed in this work, we revisit the convergence analysis of subspace correction methods~\cite{Xu:1992,XZ:2002}, the Hiptmair--Xu preconditioners~\cite{HX:2007,HX:2008}, and auxiliary grid methods~\cite{GWX:2016,Xu:1996,ZX:2014}.
All of these are analyzed in a unified and more refined manner than in existing literature, using the auxiliary space theory.
In particular, our framework provides a natural and straightforward derivation of the Xu--Zikatanov identity for semidefinite problems~\cite{LWXZ:2008,WLXZ:2008}.

The remainder of this paper is organized as follows.  
In \cref{Sec:Iterative}, we summarize the abstract theory of iterative methods for SPD linear systems.  
In \cref{Sec:Aux}, we develop an auxiliary space theory for SPD systems, including equivalences among iterative methods and analyses of error propagation operators and condition numbers.  
In \cref{Sec:Iterative_singular}, we obtain an abstract theory of iterative methods for semi-SPD linear systems as a consequence of the auxiliary space theory.
In \cref{Sec:Aux_singular}, we present the most general result, namely an auxiliary space theory for semi-SPD linear systems.  
In \cref{Sec:MSC}, we apply the auxiliary space theory to establish convergence of subspace correction methods for semi-SPD systems.  
In \cref{Sec:HX,Sec:Grid}, we reinterpret and analyze the Hiptmair--Xu preconditioners and auxiliary grid methods within the framework of the proposed auxiliary space theory.  
Finally, in \cref{Sec:Conclusion}, we conclude the paper.

\subsection{Preliminaries}
We summarize the notation used throughout this paper. 
Let $V$ be a finite-dimensional vector space equipped with an inner product $(\cdot, \cdot)$ and the induced norm $\| \cdot \|$. 
For any subspace $W \subset V$, we denote by $W^\perp$ its orthogonal complement. 
For a linear operator $T \colon V \to V$, the null space and range are denoted by $\mathcal{N}(T)$ and $\mathcal{R}(T)$, respectively. 
The adjoint of $T$ is denoted by $T^t \colon V \to V$. 

A symmetric linear operator $T \colon V \to V$ is called SPD if
\begin{equation}
\label{SPD}
(Tv, v) > 0, \quad v \in V \setminus \{ 0 \}.
\end{equation}
It is called semi-SPD if the strict inequality $>$ in~\eqref{SPD} is replaced with $\geq$.  
Equivalently,
\begin{equation*}
(Tv, v) > 0, \quad v \in \mathcal{R}(T) \setminus \{ 0 \}.
\end{equation*}


A semi-SPD linear operator $T \colon V \to V$ induces a bilinear form and an associated seminorm on $V$ given by
\[
(v, w)_T = (Tv, w), \quad |v|_T = (Tv, v)^{1/2}, \quad v, w \in V.
\]
For a linear operator $S \colon V \to V$, its operator seminorm is defined by
\begin{equation}
\label{operator_seminorm}
|S|_T = \sup_{v \in V,\ |v|_T = 1} |Sv|_T.
\end{equation}
If $T$ is SPD, then $| \cdot |_T$ is a norm on $V$, denoted by $\| \cdot \|_T$.\

Following standard convention in the literature (see, e.g.,~\cite{Xu:1992}), we use the notation $x \lesssim y$ and $y \gtrsim x$ to indicate that there exists a constant $C > 0$, independent of important parameters, such that $x \leq C y$.
We also write $x \eqsim y$ when $x \lesssim y$ and $x \gtrsim y$.

\section{Iterative methods for definite linear systems}
\label{Sec:Iterative}
For completeness, we summarize the abstract theory of iterative methods for SPD linear systems introduced in~\cite[Section~4]{XZ:2017}.
In the linear system~\eqref{model}, we assume that $A$ is SPD.
The notion of convergence for general iterative methods is defined as follows.

\begin{definition}
\label{Def:convergence}
An iterative method $\{ u^m \}$ for solving the SPD linear system~\eqref{model} is said to be convergent if $u^m$ converges to $u$ for any initial guess $u^0 \in V$.
\end{definition}

We see that the iterative method~\eqref{iterative} satisfies the following error propagation:
\begin{equation}
\label{error_propagation}
u - u^{m+1} = (I - BA)(u - u^m), \quad m \geq 0.
\end{equation}
From~\eqref{error_propagation}, we obtain the following simple criterion for the convergence of~\eqref{iterative}.

\begin{lemma}
\label{Lem:spectral_radius}
In the linear system~\eqref{model}, suppose that $A$ is SPD.
Then the iterative method~\eqref{iterative} converges if and only if
\begin{equation*}
\rho(I - BA) < 1,
\end{equation*}
where $\rho(\cdot)$ denotes the spectral radius.
\end{lemma}

To analyze the convergence rate of the iterative method~\eqref{iterative}, it is convenient to consider the following symmetrized scheme:
\begin{equation}
\label{iterative_sym} 
\begin{aligned}
u^{m+\frac{1}{2}} &= u^{m} + B(f - Au^{m}),\\
u^{m+1} &= u^{m+\frac{1}{2}} + B^t(f - Au^{m+\frac{1}{2}}),
\end{aligned}
\quad m \geq 0.
\end{equation}
The error propagation for~\eqref{iterative_sym} reads
\begin{equation*}
u - u^{m+1} = (I - B^t A)(I - B A)(u - u^{m}) = (I - \bar{B} A)(u - u^m),
\quad m \geq 0,
\end{equation*}
where the symmetrized operator $\bar{B} \colon V \to V$ is defined by
\begin{equation}
\label{symmetrized}
\bar{B} = B + B^t - B^t A B.
\end{equation}
The following theorem characterizes the convergence rate of~\eqref{iterative} in terms of the symmetrized operator $\bar{B}$~\cite[Theorem~4.1]{XZ:2017}.

\begin{theorem}
\label{Thm:error_propagation}
In the linear system~\eqref{model}, suppose that $A$ is SPD.
Then the following are equivalent:
\begin{enumerate}[(i)]
\item The symmetrized iterative method~\eqref{iterative_sym} converges.
\item The symmetrized operator $\bar{B}$ given in~\eqref{symmetrized} is SPD.
\end{enumerate}
Furthermore, if either of the above conditions holds, then the iterative method~\eqref{iterative} converges, and it satisfies
\begin{equation*}
\|I - B A\|_A^2 =
1 - \left( \sup_{v \in V,\ \|v\|_A = 1} ( \bar{B}^{-1} v, v ) \right)^{-1}.
\end{equation*}
\end{theorem}


One straightforward corollary of \cref{Thm:error_propagation} is the invertibility of $B$, as summarized in \cref{Cor:iterative}.

\begin{corollary}
\label{Cor:iterative}
In the linear system~\eqref{model}, suppose that $A$ is SPD.
If the symmetrized operator $\bar{B}$ defined in~\eqref{symmetrized} is SPD, then $B$ is nonsingular.
\end{corollary}

When both linear operators $A$ and $B$ are SPD, the $B$-preconditioned conjugate gradient method can be applied in place of the iterative method~\eqref{iterative} to solve~\eqref{model}, which is given as follows.  
\begin{equation}
\label{PCG}
\begin{aligned}
&\text{Given $u^0 \in V$, set $r^0 = f - A u^0$ and $p^0 = B r^0$. For $m \geq 0$:} \\
&\alpha_m = \frac{(r^m, r^m)}{(A p^m, p^m)}, \quad
u^{m+1} = u^m + \alpha_m p^m, \quad
r^{m+1} = r^m - \alpha_m A p^m, \\
&\beta_m = \frac{(B r^{m+1}, r^{m+1})}{(B r^m, r^m)}, \quad
p^{m+1} = B r^{m+1} + \beta_m p^m.
\end{aligned}
\end{equation}
It is well known (see, e.g.,~\cite{Saad:2003}) that~\eqref{PCG} satisfies
\begin{equation}
\label{PCG_estimate}
\| u - u^m \|_A \leq 2 \left( \frac{\sqrt{\kappa(BA)} - 1}{\sqrt{\kappa(BA)} + 1} \right)^m \| u - u^0 \|_A,
\quad m \geq 0,
\end{equation}
where the condition number $\kappa(BA)$ is given by
\begin{equation*}
\kappa(BA) = \frac{\lambda_{\max}(BA)}{\lambda_{\min}(BA)},
\end{equation*}
and $\lambda_{\min}$ and $\lambda_{\max}$ denote the minimum and maximum eigenvalues, respectively.  
Hence, in this case, it suffices to estimate the extremal eigenvalues of $BA$; see \cref{Thm:condition_number}.

\begin{theorem}
\label{Thm:condition_number}
In the iterative method~\eqref{iterative}, suppose that $A$ and $B$ are SPD.
Then we have
\begin{subequations}
\label{kappa}
\small
\begin{align}
\label{kappa_lambda_min}
\lambda_{\min}(B A)
&= \inf_{v \in V,\ \| v \|_{A} = 1} (B A v, v)_{A}
= \left( \sup_{v \in V,\ \| v \|_A = 1 } (B^{-1} v, v) \right)^{-1}, \\
\label{kappa_lambda_max}
\lambda_{\max}(B A)
&= \sup_{v \in V,\ \| v \|_{A} = 1} (B A v, v)_{A}
= \left( \inf_{v \in V,\ \| v \|_A = 1 } (B^{-1} v, v) \right)^{-1}.
\end{align}
\end{subequations}
\end{theorem}

\section{Auxiliary space theory for definite linear systems}
\label{Sec:Aux}
In this section, we present an abstract auxiliary space theory for SPD linear systems. 
Given an iterative method for solving a linear system, we construct an equivalent iterative method on an auxiliary space. 
This equivalence allows us to estimate the convergence rate and the condition number of the original method in terms of properties of the auxiliary space. 
Parts of this section have been discussed in~\cite[Section~2]{Chen:2011} and~\cite[Section~6]{XZ:2017}.

We begin by presenting \cref{Lem:aux_space_lemma}, referred to as the \textit{auxiliary space lemma}.
The idea of this lemma can be traced back to~\cite{Xu:1992}.
This result is closely related to the fictitious space method~\cite{Nepomnyaschikh:1992} and the auxiliary space method~\cite{Xu:1996}, and generalizes the additive version introduced in~\cite[Lemma~2.4]{XZ:2002}. 
It also appears in~\cite{ABMXZ:2014,Chen:2011}.

\begin{lemma}[auxiliary space lemma]
\label{Lem:aux_space_lemma}
Let $V$ and $\undertilde{V}$ be finite-dimensional vector spaces and let $\Pi \colon \undertilde{V} \to V$ be a surjective linear operator.
Let $\undertilde{B} \colon \undertilde{V} \to \undertilde{V}$ be a SPD linear operator, and define
\begin{equation*}
B = \Pi \undertilde{B} \Pi^t \colon V \to V.
\end{equation*}
Then $B$ is SPD.
Moreover, it satisfies
\begin{equation*}
( B^{-1} v, v ) = \inf_{\undertilde{v} \in \undertilde{V},\ \Pi \undertilde{v} = v } ( \undertilde{B}^{-1} \undertilde{v}, \undertilde{v} ),
\quad v \in V.
\end{equation*}
\end{lemma}
\subsection{Auxiliary system}
We consider the linear system~\eqref{model} defined on the finite-dimensional vector space $V$, where the linear operator $A$ is SPD.
Let \( \undertilde{V} \) be another finite-dimensional vector space, referred to as the \textit{auxiliary space}, and let \( \Pi \colon \undertilde{V} \to V \) be a surjective linear operator.  
The adjoint \( \Pi^t \colon V \to \undertilde{V} \) is  then injective.  
We now consider the following linear system, referred to as the \textit{auxiliary system}, defined on the auxiliary space \( \undertilde{V} \):
\begin{equation}
\label{model_aux}
\undertilde{A} \undertilde{u} = \undertilde{f},
\quad \text{where} \quad
\undertilde{A} = \Pi^t A \Pi \colon \undertilde{V} \to \undertilde{V},\
\undertilde{f} = \Pi^t f \in \undertilde{V}.
\end{equation}
We denote by $\undertilde{u} \in \undertilde{V}$ a solution to~\eqref{model_aux}.
Note that $\undertilde{A}$ is semi-SPD with $\mathcal{N} (\undertilde{A}) = \mathcal{N} (\Pi)$ and $\mathcal{R} (\undertilde{A}) = \mathcal{R} (\Pi^t)$.
The following proposition establishes the equivalence between the two linear systems~\eqref{model} and~\eqref{model_aux}.

\begin{proposition}
\label{Prop:model_equiv}
In the linear system~\eqref{model}, suppose that $A$ is SPD.
The two linear systems~\eqref{model} and~\eqref{model_aux} are equivalent in the following sense:
For $u \in V$ and $\undertilde{u} \in \undertilde{V}$, if $u$ solves~\eqref{model} and $\Pi \undertilde{u} = u$, then $\undertilde{u}$ solves~\eqref{model_aux}.
Conversely, if $\undertilde{u}$ solves~\eqref{model_aux} and $\Pi \undertilde{u} = u$, then $u$ solves~\eqref{model}.
\end{proposition}
\begin{proof}
It is clear from the injectivity of $\Pi^t$.
\end{proof}

\cref{Prop:model_equiv} shows that the two problems~\eqref{model} and~\eqref{model_aux} are equivalent.
Specifically, solving one of the systems provides a solution to the other.
Therefore, to solve~\eqref{model}, we may equivalently solve~\eqref{model_aux}, and vice versa.

\subsection{Iterative methods on the auxiliary space}
Given an iterative method for solving~\eqref{model}, we show that there exists a corresponding ``equivalent" iterative method for solving the auxiliary system~\eqref{model_aux}.
We first consider the general iterative method~\eqref{iterative}.
We also consider the following iterative method for solving the auxiliary system~\eqref{model_aux}:
\begin{subequations}
\label{iterative_aux}
\begin{equation}
\undertilde{u}^{m+1} = \undertilde{u}^{m} + \undertilde{B} (\undertilde{f} - \undertilde{A} \undertilde{u}^{m} ),
\quad m \geq 0,
\end{equation}
where $\undertilde{B} \colon \undertilde{V} \to \undertilde{V}$ is a linear operator satisfying
\begin{equation}
B = \Pi \undertilde{B} \Pi^t.
\end{equation}
\end{subequations}
The following proposition states the equivalence relation between two iterative methods~\eqref{iterative} and~\eqref{iterative_aux}.

\begin{proposition}
\label{Prop:iterative_equiv}
In the linear system~\eqref{model}, suppose that $A$ is SPD.
The iterative methods~\eqref{iterative} and~\eqref{iterative_aux} are equivalent in the following sense:
Given $u^0 \in V$, if $\{ u^m \} \subset V$ is the sequence generated by the iterative method~\eqref{iterative}, then there exists a sequence $\{ \undertilde{u}^m \} \subset \undertilde{V}$ satisfying~\eqref{iterative_aux} and
\begin{equation}
\label{Prop1:iterative_equiv}
u^m = \Pi \undertilde{u}^m,
\quad m \geq 0.
\end{equation}
Conversely, given $\undertilde{u}^0 \in \undertilde{V}$, if $\{ \undertilde{u}^m \} \subset \undertilde{V}$ is the sequence generated by the iterative method~\eqref{iterative_aux}, then the sequence $\{ u^m \} \subset V$, defined by~\eqref{Prop1:iterative_equiv}, satisfies~\eqref{iterative}.
\end{proposition}
\begin{proof}
We prove only the forward implication; the converse is immediate.
We assume that the sequence $\{ u^m \}$ is generated by the iterative method~\eqref{iterative} for a given $u^0 \in V$.
Since $\Pi$ is surjective, for each $m \geq 0$, there exists $\undertilde{\bar{u}}^m \in \undertilde{V}$ such that $u^m = \Pi \undertilde{\bar{u}}^m$.
Then~\eqref{iterative} implies
\begin{equation*}
\Pi \undertilde{\bar{u}}^{m+1} = \Pi \undertilde{\bar{u}}^{m} + B (f - A \Pi \undertilde{\bar{u}}^{m})
= \Pi \undertilde{\bar{u}}^{m} + \Pi \undertilde{B} (\undertilde{f} - \undertilde{A} \undertilde{\bar{u}}^{m}).
\end{equation*}
It follows that there exists a sequence $\{ \undertilde{\phi}^m \} \subset \mathcal{N} (\Pi)$ such that
\begin{equation*}
\undertilde{\bar{u}}^{m+1} = \undertilde{\bar{u}}^{m} + \undertilde{B} (\undertilde{f} - \undertilde{A} \undertilde{\bar{u}}^{m}) + \undertilde{\phi}^{m+1},
\quad m \geq 0.
\end{equation*}
Now, if we define
\begin{equation*}
\undertilde{u}^m = \undertilde{\bar{u}}^m - \sum_{j=1}^{m} \undertilde{\phi}^j,
\quad m \geq 0,
\end{equation*}
then it is straightforward to verify that~\eqref{iterative_aux} and~\eqref{Prop1:iterative_equiv} hold.
\end{proof}

Recall that the error propagation operator for the iterative method~\eqref{iterative} is given by $I - BA$; see~\eqref{error_propagation}.
Similarly, the error propagation operator for the iterative method~\eqref{iterative_aux} is $\undertilde{I} - \undertilde{B} \undertilde{A}$ where $\undertilde{I}$ denotes the identity operator on $\undertilde{V}$.  
In \cref{Prop:error_propagation}, we show that the operator norm of $I - BA$ coincides with the operator seminorm of $\undertilde{I} - \undertilde{B}\undertilde{A}$.
We recall that the operator seminorm $| \cdot |_T$ was defined in~\eqref{operator_seminorm}.

\begin{proposition}
\label{Prop:error_propagation}
In the linear system~\eqref{model}, suppose that $A$ is SPD.
In the iterative methods~\eqref{iterative} and~\eqref{iterative_aux}, we have
\begin{equation*}
\| I - B A \|_A = | \undertilde{I} - \undertilde{B} \undertilde{A} |_{\undertilde{A}}.
\end{equation*}
\end{proposition}
\begin{proof}
It is clear from the surjectivity of $\Pi$.
\end{proof}

In \cref{Thm:error_propagation_aux}, which follows directly from \cref{Thm:error_propagation,Lem:aux_space_lemma}, we demonstrate that the error propagation operator $I - BA$ of~\eqref{iterative} can be characterized in terms of the auxiliary system~\eqref{model_aux}.

\begin{theorem}
\label{Thm:error_propagation_aux}
In the linear system~\eqref{model}, suppose that $A$ is SPD.
If the symmetrized operator
\begin{equation}
\label{symmetrized_aux}
\undertilde{\bar{B}} = \undertilde{B} + \undertilde{B}^t - \undertilde{B}^t \undertilde{A} \undertilde{B} \colon \undertilde{V} \to \undertilde{V}
\end{equation}
is SPD, then the iterative method~\eqref{iterative} converges, and it satisfies
\begin{equation*}
\| I - B A \|_A^2 = 1 - \left( \sup_{v \in V,\ \| v \|_A = 1} \inf_{\undertilde{v} \in \undertilde{V},\ \Pi \undertilde{v} = v} (\undertilde{\bar{B}}^{-1} \undertilde{v}, \undertilde{v} ) \right)^{-1}.
\end{equation*}
\end{theorem}

Similarly, if the operator $\undertilde{B}$ is SPD, which implies by \cref{Lem:aux_space_lemma} that $B$ is also SPD, the extremal eigenvalues of $BA$ can be characterized in terms of the auxiliary system~\eqref{model_aux}; see \cref{Thm:condition_number_aux}.
We remark that \cref{Thm:condition_number_aux} can be viewed as a refined version of the classical fictitious space lemma~\cite{Nepomnyaschikh:1992,Xu:1996}.

\begin{theorem}
\label{Thm:condition_number_aux}
In the linear system~\eqref{iterative}, suppose that $A$ is SPD.
If the operator $\undertilde{B}$ given in~\eqref{iterative_aux} is SPD, then we have
\begin{align*}
\lambda_{\min} (B A) &= \left( \sup_{ v \in V,\ \| v \|_A = 1} \inf_{\undertilde{v} \in \undertilde{V},\ \Pi \undertilde{v} = v } ( \undertilde{B}^{-1} \undertilde{v}, \undertilde{v} ) \right)^{-1}, \\
\lambda_{\max} (B A) &= \left( \inf_{ v \in V,\ \| v \|_A = 1} \inf_{ \undertilde{v} \in \undertilde{V},\ \Pi \undertilde{v} = v } ( \undertilde{B}^{-1} \undertilde{v}, \undertilde{v} ) \right)^{-1}.
\end{align*}
\end{theorem}

We note that the assumptions on $\undertilde{\bar{B}}$ and $\undertilde{B}$ in \cref{Thm:error_propagation_aux,Thm:condition_number_aux}, respectively, can be relaxed.  
As stated in \cref{Thm:error_propagation_aux_general,Thm:condition_number_aux_general}, it is not necessary for these operators to be SPD on the entire space $\undertilde{V}$.  
It suffices that they are SPD on a subspace $\undertilde{W} \subseteq \undertilde{V}$ satisfying $\undertilde{W} \supset \mathcal{R}(\undertilde{A})$.

\begin{theorem}
\label{Thm:error_propagation_aux_general}
In the linear system~\eqref{model}, suppose that $A$ is SPD.
Let $\undertilde{W}$ be a subspace of $\undertilde{V}$ with $\undertilde{W} \supset \mathcal{R} (\undertilde{A})$, and let $\undertilde{Q} \colon \undertilde{V} \to \undertilde{W}$ be the orthogonal projection onto $\undertilde{W}$.
If the symmetrized operator
\begin{equation}
\label{symmetrized_aux_general}
\undertilde{\bar{B}}_{\undertilde{Q}}
= \undertilde{Q} ( \undertilde{B} + \undertilde{B}^t - \undertilde{B}^t \undertilde{A} \undertilde{B} ) \undertilde{Q}^t 
\colon \undertilde{W} \to \undertilde{W}
\end{equation}
is SPD, then the iterative method~\eqref{iterative} converges, and it satisfies
\begin{equation*}
\| I - B A \|_A^2 = 1 - \left( \sup_{v \in V,\ \| v \|_A = 1} \inf_{\undertilde{v} \in \undertilde{W},\ \Pi \undertilde{v} = v} (\undertilde{\bar{B}}_{\undertilde{Q}}^{-1} \undertilde{v}, \undertilde{v} ) \right)^{-1}.
\end{equation*}
\end{theorem}
\begin{proof}
We observe that $\Pi \undertilde{Q}^t \colon \undertilde{W} \to V$ is surjective, and that
\begin{equation*}
\bar{B} = (\Pi \undertilde{Q}^t) \undertilde{\bar{B}}_{\undertilde{Q}} (\Pi \undertilde{Q}^t)^t.
\end{equation*}
Therefore, \cref{Thm:error_propagation_aux} applies with the substitutions
\begin{equation*}
\undertilde{V} \leftarrow \undertilde{W}, 
\qquad \Pi \leftarrow \Pi \undertilde{Q}^t, 
\qquad \undertilde{B} \leftarrow \undertilde{Q} \undertilde{B} \undertilde{Q}^t,
\end{equation*}
which yields the desired result.
\end{proof}

\begin{theorem}
\label{Thm:condition_number_aux_general}
In the linear system~\eqref{iterative}, suppose that $A$ is SPD.
Let $\undertilde{W}$ be a subspace of $\undertilde{V}$ with $\undertilde{W} \supset \mathcal{R} (\undertilde{A})$, and let $\undertilde{Q} \colon \undertilde{V} \to \undertilde{W}$ be the orthogonal projection onto $\undertilde{W}$.
If the operator
\begin{equation}
\label{B_aux_general}
\undertilde{B}_{\undertilde{Q}} = \undertilde{Q} \undertilde{B} \undertilde{Q}^t \colon \undertilde{W} \to \undertilde{W}
\end{equation}
is SPD, then we have
\begin{align*}
\lambda_{\min} (B A) &= \left( \sup_{ v \in V,\ \| v \|_A = 1} \inf_{\undertilde{v} \in \undertilde{W},\ \Pi \undertilde{v} = v } ( \undertilde{B}_{\undertilde{Q}}^{-1} \undertilde{v}, \undertilde{v} ) \right)^{-1}, \\
\lambda_{\max} (B A) &= \left( \inf_{ v \in V,\ \| v \|_A = 1} \inf_{ \undertilde{v} \in \undertilde{W},\ \Pi \undertilde{v} = v } ( \undertilde{B}_{\undertilde{Q}}^{-1} \undertilde{v}, \undertilde{v} ) \right)^{-1}.
\end{align*}
\end{theorem}

\section{Iterative methods for semidefinite linear systems}
\label{Sec:Iterative_singular}
In this section, we extend the abstract theory of iterative methods from \cref{Sec:Iterative} to the semi-SPD setting. 
While related results are available in~\cite{LWXZ:2008,WLXZ:2008}, we derive the semi-SPD theory as a direct consequence of the auxiliary space framework developed in \cref{Sec:Aux}.

In the linear system~\eqref{model}, we assume that $A$ is semi-SPD.
We consider the restriction of~\eqref{model} on $\mathcal{R} (A)$, namely,
\begin{subequations}
\label{model_singular}
\begin{equation}
A_Q u_Q = f,
\end{equation}
where $Q \colon V \to \mathcal{R} (A)$ is the orthogonal projection onto $\mathcal{R} (A)$, and
\begin{equation}
A_Q = Q A Q^t \colon \mathcal{R} (A) \to \mathcal{R} (A).
\end{equation}
\end{subequations}
Since $Q^t$ is the natural embedding from $\mathcal{R}(A)$ into $V$, we omit it whenever there is no ambiguity.
Note that~\eqref{model_singular} is an SPD linear system on $\mathcal{R} (A)$.
The unique solution of~\eqref{model_singular} is denoted by $u_Q \in \mathcal{R}(A)$.

The main observation is that the original linear system~\eqref{model} can be regarded as an auxiliary system for~\eqref{model_singular}. 
Indeed, since $A = Q^t A_Q Q$, the system~\eqref{model} serves as an auxiliary system for~\eqref{model_singular} with the surjective map $Q \colon V \to \mathcal{R}(A)$. 
Therefore, as a corollary of \cref{Prop:model_equiv}, we obtain the equivalence between~\eqref{model} and~\eqref{model_singular}, as summarized in \cref{Prop:model_equiv_singular}.

\begin{proposition}
\label{Prop:model_equiv_singular}
In the linear system~\eqref{model}, suppose that $A$ is semi-SPD.
The two linear systems~\eqref{model} and~\eqref{model_singular} are equivalent in the following sense:
For $u \in V$ and $u_Q \in \mathcal{R} (A)$, if $u$ solves~\eqref{model} and $Q u = u_Q$, then $u_Q$ solves~\eqref{model_singular}.
Conversely, if $u_Q$ solves~\eqref{model_singular} and $Q u = u_Q$, then $u$ solves~\eqref{model}.
\end{proposition}

We now consider iterative methods of the form~\eqref{iterative} for the semi-SPD system~\eqref{model}. 
We begin with a precise definition of convergence for such methods. 
Since the solution to~\eqref{model} is generally nonunique, the notion of convergence in \cref{Def:convergence} is not applicable. 
Motivated by \cref{Prop:model_equiv_singular}, we introduce \cref{Def:convergence_singular} as an appropriate notion of convergence for iterative methods applied to semi-SPD linear systems.

\begin{definition}
\label{Def:convergence_singular}
An iterative method $\{ u^m \}$ for solving the semi-SPD linear system~\eqref{model} is said to be convergent if $Q u^m$ converges to $Q u$ for any initial guess $u^0 \in V$.
\end{definition}

Based on this definition, we obtain the following characterization of convergence.

\begin{proposition}
\label{Prop:convergence}
An iterative method $\{ u^m \}$ for solving the semi-SPD linear system~\eqref{model} is convergent if and only if $| u - u^m |_A$ converges to $0$ for any initial guess $u^0 \in V$.
\end{proposition}

\Cref{Prop:convergence} indicates that, to analyze the convergence of~\eqref{iterative}, 
it is sufficient to study the operator seminorm $| I - B A |_A$.

To utilize the auxiliary space theory developed in \cref{Sec:Aux}, we introduce the following iterative method for solving~\eqref{model_singular}, which can be regarded as a restriction of the iterative method~\eqref{iterative} to the range of~$A$:
\begin{subequations}
\label{iterative_singular}
\begin{equation}
u_Q^{m+1} = u_Q^{m} + B_Q (f - A_Q u_Q^{m}),
\quad m \geq 0,
\end{equation}
where $A_Q$ was given in~\eqref{model_singular}, and
\begin{equation}
B_Q = Q B Q^t \colon \mathcal{R}(A) \to \mathcal{R}(A).
\end{equation}
\end{subequations}
Thanks to \cref{Prop:iterative_equiv}, we deduce that the iterative method~\eqref{iterative} converges in the sense of \cref{Def:convergence_singular} if and only if the iterative method~\eqref{iterative_singular} converges.
Moreover, \cref{Prop:error_propagation} implies
\begin{equation}
\label{error_propagation_singular}
| I - B A |_A = \| I_Q - B_Q A_Q \|_{A_Q},
\end{equation}
where $I_Q$ denotes the identity operator on $\mathcal{R} (A)$.

Now, as a special case of \cref{Thm:error_propagation_aux}, we obtain the following sharp estimate for $| I - BA |_A$.
Applications of \cref{Thm:error_propagation_singular} to the analysis of the Gauss--Seidel method and general matrix splitting methods~\cite{Cao:2008,LWXZ:2006} can be found in~\cite{WLXZ:2008}.

\begin{theorem}
\label{Thm:error_propagation_singular}
In the linear system~\eqref{model}, suppose that $A$ is semi-SPD.
If the symmetrized operator $\bar{B}$ given in~\eqref{symmetrized} is SPD, then the iterative method~\eqref{iterative} converges, and it satisfies
\begin{equation*}
|I - B A |_A^2 =
1 - \left( \sup_{ v \in \mathcal{R} (A),\ |v|_A = 1}  \inf_{\phi \in \mathcal{N} (A) }( \bar{B}^{-1} (v + \phi), v + \phi ) \right)^{-1}.
\end{equation*}
\end{theorem}
\begin{proof}
In \cref{Thm:error_propagation_aux}, we set
\begin{equation}
\label{aux_setting_singular}
\begin{aligned}
V \leftarrow \mathcal{R} (A), \quad
\undertilde{V} \leftarrow V, \quad
\Pi \leftarrow Q, \quad
A \leftarrow A_Q, \quad
\undertilde{B} \leftarrow B.
\end{aligned}
\end{equation}
Then we obtain that the iterative method~\eqref{iterative_singular} converges, which is equivalent to that the iterative method~\eqref{iterative} converges, and that
\begin{equation*}
\| I_Q - B_Q A_Q \|_{A_Q}^2 = 1 - \left( \sup_{v_Q \in \mathcal{R} (A),\ \| v_Q \|_{A_Q} = 1} \inf_{v \in V,\ Qv = v_Q} (\bar{B}^{-1} v, v ) \right)^{-1}.
\end{equation*}
The proof is completed by invoking~\eqref{error_propagation_singular} and noting that $Qv = v_Q$ if and only if there exists $\phi \in \mathcal{N}(A)$ such that $v = v_Q + \phi$.
\end{proof}

Next, we consider the $B$-preconditioned conjugate gradient method for solving the semi-SPD linear system~\eqref{model}.
When $B$ is SPD, from~\eqref{PCG} we readily observe that if $\{ u^m \}$ is the sequence generated by the $B$-preconditioned conjugate gradient method for solving~\eqref{model}, then $\{ Q u^m \}$ coincides with the sequence generated by the $B_Q$-preconditioned conjugate gradient method for solving~\eqref{model_singular} with $u_Q^{0} = Q u^0$ (cf.~\cref{Prop:model_equiv_singular}).
Therefore,~\eqref{PCG_estimate} implies that the convergence rate of the $B$-preconditioned conjugate gradient method for solving~\eqref{model} is determined by the extremal eigenvalues of $B_Q A_Q$.
Since $B_Q A_Q = Q B A Q^t$ on $\mathcal{R}(A)$, we obtain
\begin{equation}
\label{lambda_singular}
\lambda_{\min} (B_Q A_Q) = \lambda_{\min} (BA), \quad
\lambda_{\max} (B_Q A_Q) = \lambda_{\max} (BA),
\end{equation}
where, for semi-SPD operators, $\lambda_{\min}$  denotes the smallest nonzero eigenvalue.
Hence, estimating $\lambda_{\min} (BA)$ and $\lambda_{\max} (BA)$ suffices to analyze the convergence of the $B$-preconditioned conjugate gradient method.

Analogous to \cref{Thm:error_propagation_singular}, and using~\eqref{lambda_singular}, we obtain the following estimates for $\lambda_{\min} (BA)$ and $\lambda_{\max} (BA)$ as a special case of \cref{Thm:condition_number_aux}.

\begin{theorem}
\label{Thm:condition_number_singular}
In the linear system~\eqref{model}, suppose that $A$ is semi-SPD.
If the operator $B$ given in~\eqref{iterative} is SPD, then we have
\begin{align*}
\lambda_{\min} (B A) &= \left( \sup_{ v \in \mathcal{R} (A),\ | v |_A = 1}  \inf_{\phi \in \mathcal{N} (A)} ( B^{-1} (v + \phi), v + \phi ) \right)^{-1}, \\
\lambda_{\max} (B A) &= \left( \inf_{ v \in \mathcal{R} (A),\ | v |_A = 1} \inf_{\phi \in \mathcal{N} (A)} ( B^{-1} (v + \phi), v + \phi ) \right)^{-1}.
\end{align*}
\end{theorem}
\begin{proof}
In \cref{Thm:condition_number_aux}, we adopt the setting~\eqref{aux_setting_singular}.  
By invoking~\eqref{lambda_singular}, we obtain the desired result.
\end{proof}

In view of \cref{Thm:error_propagation_aux_general,Thm:condition_number_aux_general}, the requirements in \cref{Thm:error_propagation_singular,Thm:condition_number_singular} that $\bar{B}$ and $B$ be SPD can be relaxed.  
It suffices to assume that $\bar{B}$ and $B$ are SPD only on $\mathcal{R}(A)$; see \cref{Thm:error_propagation_singular_general,Thm:condition_number_singular_general}, respectively.
Relevant discussions can be found in, e.g.,~\cite{Hayami:2018,Kaasschieter:1988}.

\begin{theorem}
\label{Thm:error_propagation_singular_general}
In the linear system~\eqref{model}, suppose that $A$ is semi-SPD.
If the symmetrized operator
\begin{equation}
\label{symmetrized_singular}
\bar{B}_Q = Q \bar{B} Q^t \colon \mathcal{R}(A) \to \mathcal{R}(A)
\end{equation} is SPD, then the iterative method~\eqref{iterative} converges, and it satisfies
\begin{equation*}
\lvert I - BA \rvert_A^2 
= 1 - \left( \sup_{v \in \mathcal{R}(A),\ \lvert v \rvert_A = 1} 
( \bar{B}_Q^{-1} v, v ) \right)^{-1}.
\end{equation*}
\end{theorem}

\begin{theorem}
\label{Thm:condition_number_singular_general}
In the linear system~\eqref{model}, suppose that $A$ is semi-SPD.
If the operator $B_Q$ given in~\eqref{iterative_singular} is SPD, then we have
\begin{align*}
\lambda_{\min} (B A) &= \left( \sup_{ v \in \mathcal{R} (A),\ | v |_A = 1} ( B_Q^{-1} v, v ) \right)^{-1}, \\
\lambda_{\max} (B A) &= \left( \inf_{ v \in \mathcal{R} (A),\ | v |_A = 1} ( B_Q^{-1} v, v ) \right)^{-1}.
\end{align*}
\end{theorem}

We note that \cref{Thm:error_propagation_singular_general} is consistent with~\cite[Lemma~2.2]{WLXZ:2008}.  
Some equivalent conditions ensuring that the operator $\bar{B}_Q$ is SPD can be found in~\cite[Lemma~2.1 and Theorem~2.1]{WLXZ:2008}.

\section{Auxiliary space theory for semidefinite linear systems}
\label{Sec:Aux_singular}
In this section, we consider the most general setting that combines those in \cref{Sec:Aux,Sec:Iterative_singular}.  
Namely, in the linear system~\eqref{model}, we assume that $A$ is semi-SPD.  
Furthermore, we assume that the operator $B$ in~\eqref{iterative} is defined in terms of the auxiliary space as in~\eqref{iterative_aux}, i.e., $B = \Pi \undertilde{B} \Pi^t$.

Within this setting, \cref{Thm:error_propagation_aux_singular} characterizes the error propagation operator $I - BA$ in terms of the auxiliary system~\eqref{model_aux}.  
Interestingly, although \cref{Thm:error_propagation_aux_singular} generalizes both \cref{Thm:error_propagation_aux,Thm:error_propagation_singular}, it can in fact be derived from \cref{Thm:error_propagation_aux}, its own special case.

\begin{theorem}
\label{Thm:error_propagation_aux_singular}
In the linear system~\eqref{model}, suppose that $A$ is semi-SPD.
If the symmetrized operator $\undertilde{\bar{B}}$ given in~\eqref{symmetrized_aux} is SPD, then the iterative method~\eqref{iterative} converges, and it satisfies
\begin{equation*}
| I - B A |_A^2 = 1 - \left( \sup_{v \in \mathcal{R} (A),\ | v |_A = 1} \inf_{\phi \in \mathcal{N} (A)} \inf_{\undertilde{v} \in \undertilde{V},\ \Pi \undertilde{v} = v + \phi} (\undertilde{\bar{B}}^{-1} \undertilde{v}, \undertilde{v} ) \right)^{-1}.
\end{equation*}
\end{theorem}
\begin{proof}
We can easily verify that $Q \Pi \colon \undertilde{V} \to \mathcal{R} (A)$ is surjective and that
\begin{equation*}
(Q \Pi) \undertilde{\bar{B}} (Q \Pi)^t = \bar{B}_Q,
\end{equation*}
where $\bar{B}_Q$ was given in~\eqref{symmetrized_singular}.
Hence, by setting
\begin{equation*}
\begin{aligned}
V \leftarrow \mathcal{R} (A), \quad
\undertilde{V} \leftarrow \undertilde{V}, \quad
\Pi \leftarrow Q \Pi , \quad
A \leftarrow A_Q, \quad
\undertilde{B} \leftarrow \undertilde{B}
\end{aligned}
\end{equation*}
in \cref{Thm:error_propagation_aux}, we deduce that the iterative method~\eqref{iterative_singular} converges.
By \cref{Prop:iterative_equiv}, this is equivalent to the convergence of~\eqref{iterative}.
Moreover, we have
\begin{equation*}
\| I_Q - B_Q A_Q \|_{A_Q}^2 = 1 - \left( \sup_{v \in \mathcal{R} (A),\ \| v \|_{A_Q} = 1} \inf_{\undertilde{v} \in \undertilde{V},\ Q \Pi \undertilde{v} = v} ( \undertilde{\bar{B}}^{-1} \undertilde{v}, \undertilde{v} ) \right)^{-1}.
\end{equation*}
The proof is completed by invoking~\eqref{error_propagation_singular} and noting that $Q \Pi \undertilde{v} = v$ if and only if there exists $\phi \in \mathcal{N}(A)$ such that $\Pi \undertilde{v} = v + \phi$.
\end{proof}

Similarly, when $\undertilde{B}$ is SPD on $\undertilde{V}$, 
\cref{Thm:condition_number_aux_singular} 
provides characterizations of the extremal eigenvalues of $BA$ in terms of the auxiliary system~\eqref{model_aux}.  
The proof is omitted, as it follow from an argument analogous to that of \cref{Thm:error_propagation_aux_singular}.

\begin{theorem}
\label{Thm:condition_number_aux_singular}
In the linear system~\eqref{iterative}, suppose that $A$ is semi-SPD.
If the operator $\undertilde{B}$ given in~\eqref{iterative_aux} is SPD, then we have
\begin{align*}
\lambda_{\min} (B A) &= \left( \sup_{ v \in \mathcal{R} (A),\ | v |_A = 1} \inf_{\phi \in \mathcal{N} (A)}\inf_{\undertilde{v} \in \undertilde{V},\ \Pi \undertilde{v} = v + \phi } ( \undertilde{B}^{-1} \undertilde{v}, \undertilde{v} ) \right)^{-1}, \\
\lambda_{\max} (B A) &= \left( \inf_{ v \in \mathcal{R} (A),\ | v |_A = 1} \inf_{\phi \in \mathcal{N} (A)} \inf_{ \undertilde{v} \in \undertilde{V},\ \Pi \undertilde{v} = v + \phi } ( \undertilde{B}^{-1} \undertilde{v}, \undertilde{v} ) \right)^{-1}.
\end{align*}
\end{theorem}

We also present counterparts of \cref{Thm:error_propagation_aux_general,Thm:condition_number_aux_general} for the semi-SPD case, given in \cref{Thm:error_propagation_aux_singular_general,Thm:condition_number_singular_general}, respectively.

\begin{theorem}
\label{Thm:error_propagation_aux_singular_general}
In the linear system~\eqref{model}, suppose that $A$ is semi-SPD.
Let $\undertilde{W}$ be a subspace of $\undertilde{V}$ with $\undertilde{W} \supset \mathcal{R} (\undertilde{A})$, and let $\undertilde{Q} \colon \undertilde{V} \to \undertilde{W}$ be the orthogonal projection onto $\undertilde{W}$.
If the symmetrized operator $\undertilde{\bar{B}}_{\undertilde{Q}}$ given in~\eqref{symmetrized_aux_general} is SPD, then the iterative method~\eqref{iterative} converges, and it satisfies
\begin{equation*}
| I - B A |_A^2 = 1 - \left( \sup_{v \in \mathcal{R} (A),\ | v |_A = 1} \inf_{\phi \in \mathcal{N} (A)} \inf_{\undertilde{v} \in \undertilde{W},\ \Pi \undertilde{v} = v + \phi} (\undertilde{\bar{B}}_{\undertilde{Q}}^{-1} \undertilde{v}, \undertilde{v} ) \right)^{-1}.
\end{equation*}
\end{theorem}

\begin{theorem}
\label{Thm:condition_number_aux_singular_general}
In the linear system~\eqref{iterative}, suppose that $A$ is semi-SPD.
Let $\undertilde{W}$ be a subspace of $\undertilde{V}$ with $\undertilde{W} \supset \mathcal{R} (\undertilde{A})$, and let $\undertilde{Q} \colon \undertilde{V} \to \undertilde{W}$ be the orthogonal projection onto $\undertilde{W}$.
If the operator $\undertilde{B}_{\undertilde{Q}}$ given in~\eqref{B_aux_general} is SPD, then we have
\begin{align*}
\lambda_{\min} (B A) &= \left( \sup_{ v \in \mathcal{R} (A),\ | v |_A = 1} \inf_{\phi \in \mathcal{N} (A)}\inf_{\undertilde{v} \in \undertilde{W},\ \Pi \undertilde{v} = v + \phi } ( \undertilde{B}_{\undertilde{Q}}^{-1} \undertilde{v}, \undertilde{v} ) \right)^{-1}, \\
\lambda_{\max} (B A) &= \left( \inf_{ v \in \mathcal{R} (A),\ | v |_A = 1} \inf_{\phi \in \mathcal{N} (A)} \inf_{ \undertilde{v} \in \undertilde{W},\ \Pi \undertilde{v} = v + \phi } ( \undertilde{B}_{\undertilde{Q}}^{-1} \undertilde{v}, \undertilde{v} ) \right)^{-1}.
\end{align*}
\end{theorem}

\begin{remark}
\label{Rem:error_propagation_aux_singular}
In \cref{Thm:error_propagation_aux_singular,Thm:error_propagation_aux_singular_general,Thm:condition_number_aux_singular,Thm:condition_number_aux_singular_general}, the characterizations involve an infimum over the null space $\mathcal{N}(A)$, a feature that also appears in the literature on iterative methods for semidefinite problems; see, e.g.,~\cite{LP:2024b,LWXZ:2008}.
\end{remark}

\section{Application I: Subspace correction methods}
\label{Sec:MSC}
As a first application of the auxiliary space theory, we consider subspace correction methods~\cite{Xu:1992}, which provide a unified framework for a broad class of iterative methods, ranging from classical Jacobi and Gauss--Seidel methods to modern domain decomposition and multigrid methods. 
In particular, we show that sharp convergence results for subspace correction methods, such as the Xu--Zikatanov identity~\cite{XZ:2002}, even in the case of semidefinite linear problems~\cite{LWXZ:2008,WLXZ:2008}, can be derived from the auxiliary space theory.
We note that a similar approach was used in~\cite{Chen:2011} to derive the Xu--Zikatanov identity for SPD linear problems.

\subsection{Space decomposition and subspace correction}
We consider the general semi-SPD linear problem~\eqref{model}. 
Our fundamental assumption is that the solution space~\( V \) admits the following decomposition:
\begin{equation}
\label{space_decomposition}
V = \sum_{j=1}^J \Pi_j V_j,
\end{equation}
where each~\( V_j \), \( 1 \leq j \leq J \), is a finite-dimensional vector space, referred to as a \textit{local space}, and \( \Pi_j \colon V_j \to V \) is a linear operator. 
In particular, if~\( V_j \) is a subspace of~\( V \) and~\( \Pi_j \) is the natural embedding, then~\eqref{space_decomposition} reduces to the standard subspace decomposition considered in much of the literature on subspace correction methods; see, e.g.,~\cite{Xu:1992,XZ:2002}.

In each local space~\( V_j \), we define the local operator~\( A_j \colon V_j \to V_j \) by
\begin{equation}
\label{A_j}
A_j = \Pi_j^t A \Pi_j.
\end{equation}
Clearly, each $A_j$ is semi-SPD.
In addition, suppose we are given a linear operator
\begin{equation*}
R_j \colon V_j \to V_j,
\end{equation*}
which acts as an approximate inverse to~\( A_j \) and serves as a local solver. 
We will later discuss sufficient conditions on~\( R_j \) that ensure convergence of the subspace correction methods.

Given a current approximation~\( u^{\mathrm{old}} \) to the solution~\( u \), the next approximation~\( u^{\mathrm{new}} \) is obtained by solving a local problem on each local space:
\begin{equation*}
u^{\mathrm{new}} = u^{\mathrm{old}} + \Pi_j R_j \Pi_j^t (f - A u^{\mathrm{old}}).
\end{equation*}
Equivalently, the update defines an error propagation operator~\( T_j \colon V \to V \) satisfying
\begin{equation}
\label{T_j}
u - u^{\mathrm{new}} = (I - T_j)(u - u^{\mathrm{old}}), 
\quad
T_j = \Pi_j R_j \Pi_j^t A.
\end{equation}

\subsection{Parallel subspace correction method}
The parallel subspace correction~(PSC) preconditioner \( B_{\mathrm{PSC}} \colon V \to V \) is given by
\begin{equation}
\label{PSC}
B_{\mathrm{PSC}} = \sum_{j=1}^J \Pi_j R_j \Pi_j^t.
\end{equation}

We first demonstrate that PSC for solving~\eqref{model} can be interpreted as a simple block iterative method for solving a certain auxiliary linear system~(cf.~\cite{Chen:2011,JPX:2025}).
We define an auxiliary space $\undertilde{V}$ and an associated linear operator $\Pi \colon \undertilde{V} \to V$ as follows:
\begin{equation}
\label{V_undertilde}
\undertilde{V} = \prod_{j=1}^J \mathcal{R} (A_j),
\quad
\Pi= [ \Pi_1, \dots, \Pi_J ].
\end{equation}
It follows from~\eqref{space_decomposition} that $\Pi$ is surjective.
We have the corresponding auxiliary linear system~\eqref{model_aux}, sometimes called \textit{expanded system}~\cite{JPX:2025}.
Clearly, the linear operator $\undertilde{A}$ has the block structure
\begin{equation}
\label{A_undertilde}
\undertilde{A} = [A_{ij}]_{i,j=1}^J,
\quad \text{where} \quad
A_{ij} = \Pi_i^t A \Pi_j, \quad 1 \leq i,j \leq J.
\end{equation}
In particular, we have $A_{jj} = A_j$.
Let $\undertilde{D}$ and $\undertilde{L}$ denote the block diagonal and
block strictly lower triangular parts of $\undertilde{A}$, respectively.
Define
\begin{equation}
\label{R_undertilde}
\undertilde{R}
= \operatorname{diag}(R_{1}, \dots, R_{J}) \colon \undertilde{V} \to \undertilde{V}.
\end{equation}

In \cref{Prop:PSC_equiv}, which directly follows from~\eqref{PSC} and~\eqref{R_undertilde}, we show that PSC for solving the original problem~\eqref{model} is equivalent to a block Jacobi method for solving the expanded system~\eqref{model_aux} in the sense of \cref{Prop:iterative_equiv}.

\begin{proposition}
\label{Prop:PSC_equiv}
The PSC preconditioner $B_{\mathrm{PSC}} \colon V \to V$ given in~\eqref{PSC} satisfies
\begin{equation*}
B_{\mathrm{PSC}} = \Pi \undertilde{B}_{\mathrm{PSC}} \Pi^t,
\quad 
\undertilde{B}_{\mathrm{PSC}} = \undertilde{R},
\end{equation*}
where $\undertilde{R}$ was given in~\eqref{R_undertilde}.
\end{proposition}

Thanks to \cref{Prop:PSC_equiv}, we are able to analyze the condition number of the preconditioned operator $B_{\mathrm{PSC}} A$ using the auxiliary space theory. 
The following theorem, which characterizes the condition number of the $B_{\mathrm{PSC}} A$, is obtained as a special case of \cref{Thm:condition_number_aux_singular_general}. 
In \cref{Thm:PSC}, for notational convenience and with a slight abuse of notation, we still denote the restriction 
\begin{subequations}
\label{R_j_restricted}
\begin{equation}
Q_j R_j Q_j^t \colon \mathcal{R}(A_j) \to \mathcal{R}(A_j)
\end{equation} 
of $R_j$ onto $\mathcal{R}(A_j)$ by 
\begin{equation}
R_j \colon \mathcal{R}(A_j) \to \mathcal{R}(A_j),
\end{equation} 
\end{subequations}
where $Q_j \colon V_j \to \mathcal{R}(A_j)$ is the orthogonal projection onto $\mathcal{R}(A_j)$.

\begin{theorem}
\label{Thm:PSC}
In the linear system~\eqref{model}, suppose that $A$ is semi-SPD.
For each $1 \leq j \leq J$, suppose that the restricted local solver $R_{j} \colon \mathcal{R}(A_j) \to \mathcal{R}(A_j)$ given in~\eqref{R_j_restricted} is SPD.
Then we have
\begin{equation}
\label{Thm1:PSC}
(B_{\mathrm{PSC}}^{-1} v, v) = \inf_{\phi \in \mathcal{N} (A)} \inf_{\substack{v_j \in \mathcal{R} (A_j),\\ \sum \Pi_j v_j = v + \phi}} \sum_{j=1}^J (R_{j}^{-1} v_j ,v_j),
\quad v \in \mathcal{R} (A),
\end{equation}
where $B_{\mathrm{PSC}}$ was given in~\eqref{PSC}.
Consequently, we have
\begin{align}
\label{Thm2:PSC}
\lambda_{\min} ( B_{\mathrm{PSC}} A ) &= \left( \sup_{ \substack{v \in \mathcal{R} (A),\\ | v |_A = 1} } \inf_{\phi \in \mathcal{N} (A)} \inf_{ \substack{ v_j \in \mathcal{R} (A_j),\\ \sum \Pi_j v_j = v + \phi} } \sum_{j=1}^J (R_{j}^{-1} v_j, v_j) \right)^{-1}, \\
\lambda_{\max} ( B_{\mathrm{PSC}} A ) &= \left( \inf_{ \substack{v \in \mathcal{R} (A),\\ | v |_A = 1} } \inf_{\phi \in \mathcal{N} (A)} \inf_{ \substack{ v_j \in \mathcal{R} (A_j),\\ \sum \Pi_j v_j = v + \phi} } \sum_{j=1}^J (R_{j}^{-1} v_j, v_j) \right)^{-1}.
\end{align}
\end{theorem}
\begin{proof}
The desired result follows by setting
\begin{equation}
\label{MSC_W}
\undertilde{W} = \mathcal{R}(\undertilde{D}) = \prod_{j=1}^J \mathcal{R}(A_j), \quad
\undertilde{Q} = \operatorname{diag}(Q_1, \dots, Q_J),
\end{equation}
in \cref{Thm:condition_number_aux_singular_general}, with \( B_{\mathrm{PSC}} \) and \( \undertilde{B}_{\mathrm{PSC}} \) given as in \cref{Prop:PSC_equiv}.
\end{proof}

\cref{Thm:PSC} implies that the $B_{\mathrm{PSC}}$-preconditioned conjugate gradient method converges provided each restricted local solver \( R_{j} \), \( 1 \leq j \leq J \), is SPD.  
The identity~\eqref{Thm1:PSC} in the SPD setting appears in~\cite[Lemma~2.4]{XZ:2002}, and its extensions to convex optimization problems are presented in~\cite{LP:2024b,Park:2020}.  
The formula~\eqref{Thm2:PSC}, often referred to as the Lions lemma~\cite{Lions:1988} or the additive Schwarz lemma, plays a central role in the analysis of parallel multilevel methods; see~\cite{Brenner:2013,TW:2005,Xu:1992} for its variants in the SPD case.

\subsection{Successive subspace correction method}
On the other hand, the successive subspace correction~(SSC) operator \( B_{\mathrm{SSC}} \colon V \to V \) is given implicitly via its associated error propagation operator:
\begin{equation}
\label{SSC}
I - B_{\mathrm{SSC}} A = (I - T_J)(I - T_{J-1}) \cdots (I - T_1),
\end{equation}
where each $T_j$ was defined in~\eqref{T_j}.
For further details, we refer the reader to~\cite{TW:2005,Xu:1992}.

\begin{remark}
\label{Rem:SSC}
Equation~\eqref{SSC} defines the SSC operator $B_{\mathrm{SSC}}$ uniquely on $\mathcal{R}(A)$, but not on $\mathcal{N}(A)$.
Nevertheless, any choice of $B_{\mathrm{SSC}}$ satisfying~\eqref{SSC} is acceptable, as it does not affect the algorithm or the analysis developed in this paper.
\end{remark}

Similar to \cref{Prop:PSC_equiv}, we can establish an equivalence between SSC and a block Gauss--Seidel method, as presented in \cref{Prop:SSC_equiv} (cf.~\cite[Theorem~3.4]{JPX:2025}).

\begin{proposition}
\label{Prop:SSC_equiv}
The SSC operator $B_{\mathrm{SSC}} \colon V \to V$ given by~\eqref{SSC} satisfies
\begin{equation*}
B_{\mathrm{SSC}} = \Pi \undertilde{B}_{\mathrm{SSC}}  \Pi^t,
\quad 
\undertilde{B}_{\mathrm{SSC}} = (\undertilde{R}^{-1}+\undertilde{L})^{-1},
\end{equation*}
where $\undertilde{R}$ was given in~\eqref{R_undertilde}, and $\undertilde{L}$ is the block strictly lower triangular part of $\undertilde{A}$ given in~\eqref{A_undertilde}.
\end{proposition}

To prove \cref{Prop:SSC_equiv}, we need the following lemma, which presents a formula for inverses of block lower triangular linear operators.

\begin{lemma}
\label{Lem:inverse_block_triangular}
Let $V_1, \dots, V_J$ be finite-dimensional vector spaces, and let $M \colon \prod_{j=1}^J V_j \to \prod_{j=1}^J V_j$ be a block lower triangular operator given by
\begin{equation*}
M = \begin{bmatrix}
M_{11} & 0 & \cdots & 0 \\
M_{21} & M_{22} & \cdots & 0 \\
\vdots & \vdots & \ddots & \vdots \\
M_{J1} & M_{J2} & \cdots & M_{JJ}
\end{bmatrix},
\end{equation*}
where each $M_{ij}$, $1 \leq i \leq J$, is a linear operator.
If each $M_{jj}$, $1 \leq j \leq J$, is nonsingular, then $M$ is nonsingular and we have
\begin{equation*}
\resizebox{\textwidth}{!}{$ \displaystyle
(M^{-1})_{ij} =
\begin{cases}
\displaystyle M_{jj}^{-1}, &
\text{ if } i = j, \\
\displaystyle \sum_{k=2}^{i-j+1} (-1)^{k-1} \sum_{j=i_1 < \dots < i_k = i} M_{i_k i_k}^{-1} M_{i_k i_{k-1}} M_{i_{k-1} i_{k-1}}^{-1} M_{i_{k-1} i_{k-2}} \dots M_{i_2 i_1} M_{i_1 i_1}^{-1}, &
\text{ if } i > j, \\
\displaystyle 0, &
\text{ if } i < j.
\end{cases}
$}
\end{equation*}
\end{lemma}
\begin{proof}
The $i$th block row of the identity $M M^{-1} = I$ yields
\begin{equation*}
(M^{-1})_{ij} = - M_{ii}^{-1} \sum_{k=j}^{i-1} M_{ik} (M^{-1})_{kj},
\quad i > j.
\end{equation*}
Applying this identity recursively gives the desired result.
\end{proof}

Now, using \cref{Lem:inverse_block_triangular}, we provide the proof of \cref{Prop:SSC_equiv} below.

\begin{proof}[Proof of \cref{Prop:SSC_equiv}]
Write $\undertilde{B}_{\mathrm{SSC}} = [B_{\mathrm{SSC}, ij}]_{i,j=1}^J$.
Since $A_{ij} = \Pi_i^t A \Pi_j$ for $1 \leq i, j \leq J$, by Lemma~\ref{Lem:inverse_block_triangular}, we have
\[
B_{\mathrm{SSC}, jj} = R_j, \quad 1 \leq j \leq J,
\]
and
\[
\resizebox{\textwidth}{!}{$
\begin{aligned}
\Pi_i B_{\mathrm{SSC}, ij} \Pi_j^t A
&= \sum_{k = 2}^{i - j + 1} (-1)^{k-1}
\sum_{j = i_1 < \dots < i_k = i}
(\Pi_{i_k} R_{i_k} \Pi_{i_k}^t A)
(\Pi_{i_{k-1}} R_{i_{k-1}} \Pi_{i_{k-1}}^t A) \cdots
(\Pi_{i_1} R_{i_1} \Pi_{i_1}^t A) \\
&= \sum_{k = 2}^{i - j + 1} (-1)^{k-1}
\sum_{j = i_1 < \dots < i_k = i}
T_{i_k} T_{i_{k-1}} \cdots T_{i_1}
\end{aligned}
$}
\]
for $1 \leq j < i \leq J$.
Hence, we obtain
\[
\begin{aligned}
\Pi \undertilde{B}_{\mathrm{SSC}} \Pi^t A
&= \sum_{j = 1}^J T_j
+ \sum_{1 \leq j < i \leq J} \sum_{k = 2}^{i - j + 1} (-1)^{k-1}
\sum_{j = i_1 < \dots < i_k = i}
T_{i_k} T_{i_{k-1}} \cdots T_{i_1} \\
&= I - (I - T_J)(I - T_{J-1}) \cdots (I - T_1).
\end{aligned}
\]
Since $I - B_{\mathrm{SSC}} A = (I - T_J)(I - T_{J-1}) \cdots (I - T_1)$, we obtain the desired result.
\end{proof}

\begin{remark}
\label{Rem:SSC_equiv}
Despite \cref{Prop:SSC_equiv} is stated in terms of the inverse $\undertilde{R}^{-1}$, it remains valid even when some $R_i$ are not nonsingular. 
Indeed, the proof of \cref{Prop:SSC_equiv} does not rely on the invertibility of $\undertilde{R}$. 
In cases where certain $R_i$ are singular, it suffices to define the block Gauss--Seidel operator $(\undertilde{R}^{-1} + \undertilde{L})^{-1}$ formally in the sense of \cref{Lem:inverse_block_triangular}.
\end{remark}

Let \( P_j \colon V \to \mathcal{R}(A_j) \) be the linear operator satisfying
\begin{equation}
\label{MSC_P_j}
( P_j v, v_j )_{A_{j}} = ( v, \Pi_j v_j )_{A},
\quad v \in V,\ v_j \in \mathcal{R}(A_j).
\end{equation}
Then we can readily verify the following properties of $P_j$:
\begin{equation}
\label{MSC_P_j_properties}
A_{j} P_j = Q_j \Pi_j^t A, \quad
P_j \Pi_j Q_j^t = I_{Q_j},
\end{equation}
where $I_{Q_j}$ denotes the identity operator on $\mathcal{R} (A_j)$.

The following theorem, widely known as the Xu--Zikatanov identity~\cite{XZ:2002} (see also~\cite{LWXZ:2008,WLXZ:2008} for the semi-SPD case), provides a sharp estimate for $| I - B A |_A$ in SSC. 
In contrast to the proofs for the semi-SPD case in~\cite{LWXZ:2008,WLXZ:2008}, our approach does not rely on advanced mathematical tools such as pseudoinverses.
In \cref{Thm:XZ}, as in \cref{Thm:PSC}, we use the notation $R_j \colon \mathcal{R}(A_j) \to \mathcal{R}(A_j)$ for the restricted local solver~\eqref{R_j_restricted}.

\begin{theorem}[Xu--Zikatanov identity]
\label{Thm:XZ}
In the linear system~\eqref{model}, suppose that $A$ is semi-SPD.
For each $1 \leq j \leq J$, suppose that the symmetrized restricted local solver
\begin{equation*}
\bar{R}_j = R_j + R_j^t - R_j^t A_j R_j \colon \mathcal{R} (A_j) \to \mathcal{R} (A_j)
\end{equation*}
is SPD, where $R_j$ was given in~\eqref{R_j_restricted}.
Then SSC converges, and it satisfies
\begin{subequations}
\label{Thm1:XZ}
\begin{equation}
\label{Thm1:XZ_1}
| I - B_{\mathrm{SSC}} A |_A^2
= 1-\frac{1}{1+c_0}
= 1-\frac{1}{c_1},
\end{equation}
where $B_{\mathrm{SSC}}$ was given in~\eqref{SSC}, the constants \( c_0 \) and \( c_1 \) are defined by
\begin{equation}
\label{Thm1:XZ_2}
\begin{aligned}
&c_0  = \sup_{ \substack{ v \in \mathcal{R}(A), \\ |v|_A=1 } }  \inf_{\phi \in \mathcal{N} (A) } \inf_{ \substack{ v_j \in \mathcal{R} (A_j),\\ \sum \Pi_j v_j = v + \phi } } \sum_{j=1}^J \| R_{j}^t w_j \|_{\bar{R}_{j}^{-1}}^{2}, \\
&c_1 = \sup_{ \substack{ v \in \mathcal{R}(A), \\ |v|_A=1 } } \inf_{\phi \in \mathcal{N} (A)} \inf_{ \substack{ v_j \in \mathcal{R} (A_j),\\ \sum \Pi_j v_j = v + \phi } } \sum_{j=1}^J
\| \bar{R}_{j} R_{j}^{-1} v_j + R_{j}^t  w_j \|_{\bar{R}_{j}^{-1}}^2, \\
&\text{with} \quad
w_i = A_{i} P_i \sum_{j = i}^J \Pi_j v_j -R_{i}^{-1} v_i,
\quad 1 \leq i \leq J,
\end{aligned}
\end{equation}
and $P_j \colon V \to \mathcal{R} (A_j)$ was defined in~\eqref{MSC_P_j}.
In particular, when $R_{j} = A_{j}^{-1}$ for every $1 \leq j \leq J$, we have 
\begin{equation}
\label{Thm1:XZ_3}
\begin{aligned}
c_0 &= \sup_{ \substack{ v \in \mathcal{R}(A),\\ |v|_A=1 } } \inf_{\phi \in \mathcal{N} (A)} \inf_{ \substack{ v_j \in \mathcal{R} (A_j), \\ \sum \Pi_j v_j = v + \phi } } \sum_{i=1}^J \left\| P_i \sum_{j = i+1}^J \Pi_j v_j \right\|_{A_i}^{2}, \\
c_1 &= \sup_{ \substack{ v \in \mathcal{R}(A),\\ |v|_A=1 } } \inf_{\phi \in \mathcal{N} (A)} \inf_{ \substack{ v_j \in \mathcal{R} (A_j),\\ \sum \Pi_j v_j = v + \phi } } \sum_{i=1}^J
\left\| P_i \sum_{j = i}^J \Pi_j v_j \right\|_{A_i}^2.
\end{aligned}
\end{equation}
\end{subequations}
\end{theorem}
\begin{proof}
To prove the convergence of SSC, it suffices, by \cref{Thm:error_propagation_aux_singular_general}, to show that the operator 
\(\undertilde{\bar{B}}_{\mathrm{SSC},\undertilde{Q}} \colon \undertilde{W} \to \undertilde{W}\), defined analogously to~\eqref{symmetrized_aux_general}, is SPD, where $\undertilde{W}$ and $\undertilde{Q}$ are defined in~\eqref{MSC_W}.  
With a slight abuse of notation, we drop the subscript $\undertilde{Q}$ and write $\undertilde{\bar{B}}_{\mathrm{SSC}} \colon \undertilde{W} \to \undertilde{W}$ whenever no ambiguity arises.

By \cref{Cor:iterative}, the definiteness of each \( \bar{R}_{j} \) implies the invertibility of each \( R_{j} \), and hence \( \undertilde{R} = \operatorname{diag} (R_1, \dots, R_J) \colon \undertilde{W} \to \undertilde{W} \) is nonsingular. Then, by \cref{Lem:inverse_block_triangular}, \( \undertilde{B}_{\mathrm{SSC}} \colon \undertilde{W} \to \undertilde{W} \) is also nonsingular.
By direct calculation, we obtain
\begin{equation*}
\undertilde{\bar{B}}_{\mathrm{SSC}} 
= \undertilde{B}_{\mathrm{SSC}}^t \undertilde{R}^{-t} \undertilde{\bar{R}} \undertilde{R}^{-1} \undertilde{B}_{\mathrm{SSC}},
\end{equation*}
where
\begin{equation*}
    \undertilde{\bar{R}} = \operatorname{diag} (\bar{R}_1, \dots, \bar{R}_J) = \undertilde{R} + \undertilde{R}^t - \undertilde{R}^t \undertilde{D} \undertilde{R} \colon \undertilde{W} \to \undertilde{W}.
\end{equation*}
Hence, the definiteness of \( \undertilde{\bar{R}} \), which is equivalent to the definiteness of each \( \bar{R}_{j} \), implies that \( \undertilde{\bar{B}}_{\mathrm{SSC}} \) is SPD. Consequently, SSC is convergent.

Next, we prove~\eqref{Thm1:XZ}.
By \cref{Thm:error_propagation_aux_singular_general}, we have
\begin{equation}
\label{Thm2:XZ}
| I - B_{\mathrm{SSC}} A |_A^2 = 1 - \left( \sup_{ \substack{ v \in \mathcal{R} (A),\\ | v |_A = 1 }} \inf_{\phi \in \mathcal{N} (A)} \inf_{ \substack{ \undertilde{v} \in \undertilde{W},\\ \Pi \undertilde{v} = v + \phi } } (\undertilde{\bar{B}}_{\mathrm{SSC}}^{-1} \undertilde{v}, \undertilde{v} ) \right)^{-1}.
\end{equation}
Since $\undertilde{R}$ and $\undertilde{\bar{R}}$ are nonsingular, a calculation similar to that in~\cite[Theorem~3]{Chen:2011} yields
\begin{equation*}
\undertilde{\bar{B}}_{\mathrm{SSC}}^{-1} 
= (\undertilde{\bar{R}} \undertilde{R}^{-1} + \undertilde{R}^{t} \undertilde{\hat{L}}^t )^t
\undertilde{\bar{R}}^{-1}
(\undertilde{\bar{R}} \undertilde{R}^{-1} + \undertilde{R}^{t} \undertilde{\hat{L}}^t ) 
= \undertilde{A} + \undertilde{\hat{L}} \undertilde{R} \undertilde{\bar{R}}^{-1}
\undertilde{R}^t \undertilde{\hat{L}}^t,
\end{equation*}
where
\begin{equation*}
\undertilde{\hat{L}} = \undertilde{Q} ( \undertilde{L} + \undertilde{D} ) \undertilde{Q}^t - \undertilde{R}^{-t} \colon \undertilde{W} \to \undertilde{W}.
\end{equation*}
Hence, we get
\begin{equation}
\label{Thm3:XZ}
(\undertilde{\bar{B}}_{\mathrm{SSC}}^{-1} \undertilde{v}, \undertilde{v} )
= \| \bar{\undertilde{R}} \undertilde{R}^{-1} \undertilde{v} + \undertilde{R}^t \undertilde{\hat{L}}^t \undertilde{v} \|_{\bar{\undertilde{R}}^{-1}}^2
= | \Pi \undertilde{v} |_A^2 + \| \undertilde{R}^t \undertilde{\hat{L}}^t \undertilde{v} \|_{\bar{\undertilde{R}}^{-1}}^2,\ \undertilde{v} \in \undertilde{W},
\end{equation}
We observe that, for each $1 \leq i \leq J$, the $i$th component $w_i$ of $\undertilde{\hat{L}}^t \undertilde{v}$ satisfies
\begin{equation}
\label{Thm4:XZ}
w_i
= Q_i \Pi_i^t A \sum_{j = i}^J \Pi_j v_j - R_{i}^{-1} v_i
= A_{i} P_i \sum_{j = i}^J \Pi_j v_j - R_{i}^{-1} v_i.
\end{equation}
where the last equality follows from~\eqref{MSC_P_j_properties}.
Combining~\eqref{Thm2:XZ},~\eqref{Thm3:XZ}, and~\eqref{Thm4:XZ} yields~\eqref{Thm1:XZ_1} and~\eqref{Thm1:XZ_2}.

To prove~\eqref{Thm1:XZ_3}, observe that, if $R_{i} = A_{i}^{-1}$, then
\begin{equation*}
R_{i}^t w_i = P_{i} \sum_{j=i}^J \Pi_j v_j - v_i
= P_{i} \sum_{j=i+1}^J \Pi_j v_j,
\end{equation*}
where the last inequality is due to~\eqref{MSC_P_j_properties}.
This completes the proof.
\end{proof}

In view of \cref{Thm:error_propagation_singular_general}, \cref{Thm:XZ} shows that SSC converges provided that each restricted local solver \( R_{j} \) is convergent. Equivalent conditions for the convergence of \( R_{j} \) can be found in~\cite{Brenner:2013,LWXZ:2008,WLXZ:2008}.

We also remark that several different proofs of the Xu--Zikatanov identity in the SPD case have appeared in the literature; see~\cite{Brenner:2013,Chen:2011,CXZ:2008,XZ:2002}. Applications of the Xu--Zikatanov identity to the analysis of various multilevel iterative methods can be found in~\cite{Brenner:2013,LWXZ:2008,XZ:2002}.

\section{Application II: Hiptmair--Xu preconditioners}
\label{Sec:HX}
The Hiptmair--Xu preconditioners~\cite{HX:2007,HX:2008} refer to a class of optimal preconditioners for \( H(\operatorname{curl}) \)- and \( H(\operatorname{div}) \)-problems based on the auxiliary space method~\cite{Xu:1996}.
Some recent developments can be found in~\cite{Hu:2021,Hu:2023}.
In this section, we present a refined convergence analysis of the Hiptmair--Xu preconditioners based on the auxiliary space theory developed in this paper.

\subsection{Preliminaries}
Let \( \Omega \subset \mathbb{R}^3 \) be a bounded polyhedral domain.  
We consider the following model variational problem: find \( u \in H_0 (\mathsf{D}; \Omega) \) such that  
\begin{equation}
\label{HX_model_continuous}
(\mathsf{D} u, \mathsf{D} v)_{L^2} + \epsilon (u, v)_{L^2} = f(v),
\quad v \in H_0 (\mathsf{D}; \Omega),
\end{equation}
where \( \mathsf{D} \) denotes a differential operator such as \( \operatorname{grad} \), \( \operatorname{curl} \), or \( \operatorname{div} \), \( f \) is a linear functional on \( H_0 (\mathsf{D}; \Omega) \), and $\epsilon$ is a nonnegative parameter.
For simplicity, we assume that $\epsilon \in [0,1]$; the case $\epsilon > 1$ was discussed in~\cite{HX:2008}.

To discretize~\eqref{HX_model_continuous}, we consider a quasi-uniform triangulation \( \mathcal{T}_h \) of \( \Omega \) and the corresponding lowest-order conforming finite element space \( H_h(\mathsf{D}) \), where $h \in L^{\infty} (\Omega)$ stands for the piecewise constant mesh diameter function; see~\cite{Hiptmair:2002,HX:2007} for details.  
Applying the Galerkin method to~\eqref{HX_model_continuous} using the space $H_h(\mathsf{D})$, we obtain a linear system of the form
\begin{equation}
\label{HX_model}
A_{\mathsf{D}} u = f.
\end{equation}
In~\eqref{HX_model}, $A_{\mathsf{D}}$ is semi-SPD when $\epsilon = 0$ and $\mathsf{D}$ is either $\operatorname{curl}$ or $\operatorname{div}$, and is SPD otherwise.

For notational convenience, we define
\begin{equation*}
\mathsf{D}^0 = \operatorname{grad}, \quad
\mathsf{D}^1 = \operatorname{curl}, \quad
\mathsf{D}^2 = \operatorname{div}, \quad
\mathsf{D}^3 = I.
\end{equation*}
In addition, we let \(S_h\) denote the space of continuous and piecewise linear functions in either $H_0^1 (\Omega)$ or $H_0^1 (\Omega)^3$, depending on the context.

As discussed in~\cite{HX:2007}, the finite element space \( H_h(\mathsf{D}) \) admits a canonical interpolation operator \( \Pi_h^{\mathsf{D}} \colon C(\Omega) \to H_h(\mathsf{D}) \), which assigns a unique finite element function matching the same nodal values.  
The canonical interpolation operator \( \Pi_h^{\mathsf{D}^k} \) (\( k = 0,1,2 \)) satisfies the following properties:
\begin{equation}
\label{canonical_interpolation}
\mathsf{D}^k \Pi_h^{\mathsf{D}^k} v_h = \Pi_h^{\mathsf{D}^{k+1}} \mathsf{D}^k v_h, \quad
\| \Pi_h^{\mathsf{D}^k} v_h \|_{L^2} \lesssim \| v_h \|_{L^2}, \quad v_h \in S_h.
\end{equation}

The most important properties of the spaces \( H_h(\operatorname{curl}) \) and \( H_h(\operatorname{div}) \) for constructing the Hiptmair--Xu preconditioners are the so-called discrete regular decompositions~\cite{HX:2007,Hu:2021}.  
For any \( v_h \in H_h(\operatorname{curl}; \Omega) \), there exist \( \tilde{v}_h \in H_h(\operatorname{curl}; \Omega) \), \( \psi_h \in H_h(\operatorname{grad}; \Omega)^3 \), and \( p_h \in H_h(\operatorname{grad}; \Omega) \) such that
\begin{equation}
\label{regular_decomposition_curl}
\begin{aligned}
&v_h = \tilde{v}_h + \Pi_h^{\operatorname{curl}} \psi_h + \operatorname{grad} p_h, \\
&\| h^{-1} \tilde{v}_h \|_{L^2}^2 + \| \psi_h \|_{H^1}^2 \lesssim \| \operatorname{curl} v_h \|_{L^2}^2, \quad
\| p_h \|_{H^1}^2 \lesssim \| v_h \|_{H(\operatorname{curl})}^2.
\end{aligned}
\end{equation}
Similarly, for any \( v_h \in H_h(\operatorname{div}; \Omega) \), there exist \( \tilde{v}_h \in H_h(\operatorname{div}; \Omega) \), \( \psi_h \in H_h(\operatorname{grad}; \Omega)^3 \), and \( w_h \in H_h(\operatorname{curl}; \Omega) \) such that
\begin{equation}
\label{regular_decomposition_div}
\begin{aligned}
&v_h = \tilde{v}_h + \Pi_h^{\operatorname{div}} \psi_h + \operatorname{curl} w_h, \\
&\| h^{-1} \tilde{v}_h \|_{L^2}^2 + \| \psi_h \|_{H^1}^2 \lesssim \| \operatorname{div} v_h \|_{L^2 }^2, \quad
\| w_h  \|_{H(\operatorname{curl})}^2 \lesssim \| v_h \|_{H(\operatorname{div})}^2.
\end{aligned}
\end{equation}

\subsection{\texorpdfstring{$H(\operatorname{curl})$}{H(curl)} problems}
For $\epsilon > 0$, the Hiptmair--Xu preconditioner \( B_{\operatorname{curl}} \) for the linear operator \( A_{\operatorname{curl}} \) given in~\eqref{HX_model} with $\mathsf{D} = \operatorname{curl}$ is defined by
\begin{equation}
\label{HX_curl}
B_{\operatorname{curl}} =
D_{\operatorname{curl}}^{-1} +
\Pi_h^{\operatorname{curl}} A_{\operatorname{grad}^3}^{-1} (\Pi_h^{\operatorname{curl}})^t +
\epsilon^{-1} \operatorname{grad} A_{\operatorname{grad}}^{-1} \operatorname{grad}^t,
\end{equation}
where \( D_{\operatorname{curl}}^{-1} \) is the nodal Jacobi smoother.  
We observe that the structure of \( B_{\operatorname{curl}} \) is consistent with the auxiliary space preconditioner given in~\eqref{iterative_aux}.  
Namely, we have
\begin{equation}
\label{HX_curl_aux}
\begin{aligned}
&B_{\operatorname{curl}} = \Pi \undertilde{B}_{\operatorname{curl}} \Pi^t, \quad \text{with} \\
&V = H_h (\operatorname{curl}), \quad
\undertilde{V} = H_h(\operatorname{curl}) \times H_h(\operatorname{grad})^3 \times H_h(\operatorname{grad}), \\
&\Pi = \left[ I,\ \Pi_h^{\operatorname{curl}},\ \operatorname{grad} \right], \quad
\undertilde{B}_{\operatorname{curl}} =
\begin{bmatrix}
D_{\operatorname{curl}}^{-1} & 0 & 0 \\
0 & A_{\operatorname{grad}^3}^{-1} & 0 \\
0 & 0 & \epsilon^{-1} A_{\operatorname{grad}}^{-1}
\end{bmatrix}.
\end{aligned}
\end{equation}

On the other hand, for the case $\epsilon = 0$, the last term in~\eqref{HX_curl} can be dropped.
That is, we define
\begin{equation}
\label{HX_curl_singular}
B_{\operatorname{curl}}
= D_{\operatorname{curl}}^{-1}
+ \Pi_h^{\operatorname{curl}} A_{\operatorname{grad}^3}^{-1} (\Pi_h^{\operatorname{curl}})^t,
\end{equation}
so that we have the following structure of the auxiliary space preconditioner~\eqref{iterative_aux}:
\begin{equation}
\label{HX_curl_singular_aux}
\begin{aligned}
&B_{\operatorname{curl}} = \Pi \undertilde{B}_{\operatorname{curl}} \Pi^t, \quad \text{with} \\
&V = H_h (\operatorname{curl}), \quad
\undertilde{V} = H_h(\operatorname{curl}) \times H_h(\operatorname{grad})^3, \\
&\Pi = \left[ I,\ \Pi_h^{\operatorname{curl}} \right], \quad
\undertilde{B}_{\operatorname{curl}} =
\begin{bmatrix}
D_{\operatorname{curl}}^{-1} & 0 \\
0 & A_{\operatorname{grad}^3}^{-1}
\end{bmatrix}.
\end{aligned}
\end{equation}

Hence, the condition number \( \kappa(B_{\operatorname{curl}} A_{\operatorname{curl}}) \) can be estimated using \cref{Thm:condition_number_aux,Thm:condition_number_aux_singular}. 
In \cref{Thm:HX_curl}, we prove the optimality of the Hiptmair--Xu preconditioner; namely, \( \kappa(B_{\operatorname{curl}} A_{\operatorname{curl}}) \) is bounded by a uniform constant~\cite[Theorem~7.1]{HX:2007}.

\begin{theorem}
\label{Thm:HX_curl}
The Hiptmair--Xu preconditioner $B_{\operatorname{curl}}$ defined in~\eqref{HX_curl} and~\eqref{HX_curl_singular} satisfies
\begin{equation*}
\kappa (B_{\operatorname{curl}} A_{\operatorname{curl}}) \lesssim 1.
\end{equation*}
\end{theorem}
\begin{proof}
We first consider the case $\epsilon > 0$, i.e., when $A_{\operatorname{curl}}$ is SPD.
Note that~\eqref{HX_curl_aux} implies
\begin{equation*}
\| \undertilde{v}_h \|_{\undertilde{B}_{\operatorname{curl}}^{-1}}^2
= \| \tilde{v}_h \|_{D_{\operatorname{curl}}}^2 + \| \psi_h \|_{A_{\operatorname{grad}^3}}^2 + \epsilon \| p_h\|_{A_{\operatorname{grad}}}^2,
\quad \undertilde{v}_h = (\tilde{v}_h, \psi_h, p_h) \in \undertilde{V}.
\end{equation*}
Thanks to \cref{Thm:condition_number_aux}, it suffices to verify the following:
\begin{enumerate}[(a)]
\item For any $\tilde{v}_h \in H_h (\operatorname{curl})$, $\psi_h \in H_h (\operatorname{grad})^3$, and $p_h \in H_h (\operatorname{grad})$, we have
\begin{equation*}
\| \tilde{v}_h + \Pi_h^{\operatorname{curl}} \psi_h + \operatorname{grad} p_h \|_{A_{\operatorname{curl}}}^2
\lesssim \| \tilde{v}_h \|_{D_{\operatorname{curl}}}^2 + \| \psi_h \|_{A_{\operatorname{grad}^3}}^2 + \epsilon \| p_h\|_{A_{\operatorname{grad}}}^2.
\end{equation*}
\item For any $v_h \in H_h (\operatorname{curl})$, there exist $\tilde{v}_h \in H_h (\operatorname{curl})$, $\psi_h \in H_h (\operatorname{grad})^3$, and $p_h \in H_h (\operatorname{grad})$ such that $v_h = \tilde{v}_h + \Pi_h^{\operatorname{curl}} \psi_h + \operatorname{grad} p_h$ and
\begin{equation*}
\| \tilde{v}_h \|_{D_{\operatorname{curl}}}^2 + \| \psi_h \|_{A_{\operatorname{grad}^3}}^2 + \epsilon \| p_h\|_{A_{\operatorname{grad}}}^2
\lesssim \| v_h \|_{A_{\operatorname{curl}}}^2.
\end{equation*}
\end{enumerate}

To prove~(a), take any $\tilde{v}_h \in H_h(\operatorname{curl})$, $\psi_h \in H_h(\operatorname{grad})^3$, and $p_h \in H_h(\operatorname{grad})$.  
By the standard coloring argument, we have the estimate
\begin{equation}
\label{Thm1:HX_curl}
\| \tilde{v}_h \|_{A_{\operatorname{curl}}}^2
\lesssim \| \tilde{v}_h \|_{D_{\operatorname{curl}}}^2.
\end{equation}
On the other hand, by the properties of the canonical interpolation operators presented in~\eqref{canonical_interpolation}, we deduce
\begin{equation}
\label{Thm2:HX_curl}
\begin{aligned}
\| \Pi_h^{\operatorname{curl}} \psi_h \|_{A_{\operatorname{curl}}}^2
&= \| \operatorname{curl} \Pi_h^{\operatorname{curl}} \psi_h \|_{L^2}^2 + \epsilon \| \Pi_h^{\operatorname{curl}} \psi_h \|_{L^2}^2 \\
&= \| \Pi_h^{\operatorname{div}} \operatorname{curl} \psi_h \|_{L^2}^2 + \epsilon \| \Pi_h^{\operatorname{curl}} \psi_h \|_{L^2}^2 \\
&\lesssim \| \operatorname{curl} \psi_h \|_{L^2}^2 + \epsilon \| \psi_h \|_{L^2}^2 \\
&\lesssim  \| \operatorname{grad} \psi_h \|_{L^2}^2 + \epsilon \| \psi_h \|_{L^2}^2\\
&= \| \psi_h \|_{A_{\operatorname{grad}^3}}^2.
\end{aligned}
\end{equation}
Additionally, we directly obtain
\begin{equation}
\label{Thm3:HX_curl}
\| \operatorname{grad} p_h \|_{A_{\operatorname{curl}}}^2 
= \epsilon \| \operatorname{grad} p_h \|_{L^2}^2 
\leq \epsilon \| p_h \|_{A_{\operatorname{grad}}}^2.
\end{equation}
Combining the estimates~\eqref{Thm1:HX_curl}, \eqref{Thm2:HX_curl}, and~\eqref{Thm3:HX_curl}, we arrive at
\begin{align*}
\big\| \tilde{v}_h + \Pi_h^{\operatorname{curl}} \psi_h + \operatorname{grad} p_h \big\|_{A_{\operatorname{curl}}}^2 
&\lesssim \| \tilde{v}_h \|_{A_{\operatorname{curl}}}^2 + \| \Pi_h^{\operatorname{curl}} \psi_h \|_{A_{\operatorname{curl}}}^2 + \| \operatorname{grad} p_h \|_{A_{\operatorname{curl}}}^2 \\
&\lesssim \| \tilde{v}_h \|_{D_{\operatorname{curl}}}^2 + \| \psi_h \|_{A_{\operatorname{grad}^3}}^2 + \epsilon  \| p_h \|_{A_{\operatorname{grad}}}^2,
\end{align*}
which completes the proof of~(a).

Next, we prove~(b).
Given $v_h \in H_h (\operatorname{curl})$, let $\tilde{v}_h \in H_h (\operatorname{curl})$, $\psi_h \in H_h (\operatorname{grad})^3$, and $p_h \in H_h (\operatorname{grad})$ be given by the discrete regular decomposition~\eqref{regular_decomposition_curl}.
Then we get
\begin{align*}
\| h^{-1} \tilde{v}_h \|_{L^2}^2 + \| \psi_h \|_{A_{\operatorname{grad}^3}}^2
\leq \| h^{-1} \tilde{v}_h \|_{L^2}^2 + \| \psi_h \|_{H^1}^2
\lesssim \| \operatorname{curl} v_h \|_{L^2}^2
\leq \| v_h \|_{A_{\operatorname{curl}}}^2, \\
\epsilon \| p_h \|_{A_{\operatorname{grad}}}^2
\leq \epsilon \| p_h \|_{H^1}^2
\lesssim \epsilon \| v_h \|_{H(\operatorname{curl})}^2
\leq \| v_h \|_{A_{\operatorname{curl}}}^2.
\end{align*}
Hence, we only need to prove that
\begin{equation*}
\| \tilde{v}_h \|_{D_{\operatorname{curl}}}^2 \lesssim \| h^{-1} \tilde{v}_h \|_{L^2}^2,
\end{equation*}
which is a direct consequence of the inverse inequality and the finite overlap property.
This completes the proof of the case $\epsilon > 0$.

Next, we consider the case $\epsilon = 0$, i.e., when $A_{\operatorname{curl}}$ is semi-SPD.
Thanks to \cref{Thm:condition_number_aux_singular} and~\eqref{HX_curl_singular_aux}, it suffices to verify the following:
\begin{enumerate}[(a)$'$]
\item For any $\tilde{v}_h \in H_h (\operatorname{curl})$, $\psi_h \in H_h (\operatorname{grad})^3$, and $\phi_h \in \mathcal{N} (\operatorname{curl})$, we have
\begin{equation*}
| \tilde{v}_h + \Pi_h^{\operatorname{curl}} \psi_h - \phi_h |_{A_{\operatorname{curl}}}^2
\lesssim \| \tilde{v}_h \|_{D_{\operatorname{curl}}}^2 + \| \psi_h \|_{A_{\operatorname{grad}^3}}^2.
\end{equation*}
\item For any $v_h \in H_h (\operatorname{curl})$, there exist $\tilde{v}_h \in H_h (\operatorname{curl})$, $\psi_h \in H_h (\operatorname{grad})^3$, and $\phi_h \in \mathcal{N} (\operatorname{curl})$ such that $v_h + \phi_h = \tilde{v}_h + \Pi_h^{\operatorname{curl}} \psi_h$ and
\begin{equation*}
\| \tilde{v}_h \|_{D_{\operatorname{curl}}}^2 + \| \psi_h \|_{A_{\operatorname{grad}^3}}^2
\lesssim | v_h |_{A_{\operatorname{curl}}}^2.
\end{equation*}
\end{enumerate}
We readily observe that (a)$'$ can be established by the same argument as (a).  
To prove (b)$'$, given $v_h \in H_h(\operatorname{curl})$, let 
$\tilde{v}_h \in H_h(\operatorname{curl})$, 
$\psi_h \in H_h(\operatorname{grad})^3$, and 
$p_h \in H_h(\operatorname{grad})$ 
be given by the discrete regular decomposition~\eqref{regular_decomposition_curl}.  
We further set $\phi_h = - \operatorname{grad} p_h$.  
Then the same argument as in (b) applies, which completes the proof.
\end{proof}

\begin{remark}
\label{Rem:smoother_grad}
To construct more practical preconditioners, we may replace \( A_{\operatorname{grad}^3}^{-1} \) and \( A_{\operatorname{grad}}^{-1} \) in~\eqref{HX_curl} with suitable optimal preconditioners, such as the Bramble--Pasciak--Xu preconditioner~\cite{BPX:1990,LX:2016,Zhang:1992} or algebraic multigrid preconditioners~\cite{KV:2009,KV:2012}.  
The convergence analysis remains valid in this case as well.
\end{remark}

\subsection{\texorpdfstring{$H(\operatorname{div})$}{H(div)} problems}
For $\epsilon > 0$, the Hiptmair--Xu preconditioner \( B_{\operatorname{div}} \) for the linear operator \( A_{\operatorname{div}} \) given in~\eqref{HX_model} with $\mathsf{D} = \operatorname{div}$ is defined by
\begin{equation}
\label{HX_div}
B_{\operatorname{div}} =
D_{\operatorname{div}}^{-1} +
\Pi_h^{\operatorname{div}} A_{\operatorname{grad}^3}^{-1} (\Pi_h^{\operatorname{div}})^t +
\epsilon^{-1} \operatorname{curl} A_{\operatorname{curl}}^{-1} \operatorname{curl}^t,
\end{equation}
where \( D_{\operatorname{div}}^{-1} \) is the nodal Jacobi smoother.  
Similar to the case of \( H(\operatorname{curl}) \) problems, the structure of \( B_{\operatorname{div}} \) is consistent with the auxiliary space preconditioner~\eqref{iterative_aux}:
\begin{equation*}
\begin{aligned}
&B_{\operatorname{div}} = \Pi_{\operatorname{div}} \undertilde{B}_{\operatorname{div}} \Pi_{\operatorname{div}}^t, \quad \text{with} \\
&V = H_h (\operatorname{div}), \quad
\undertilde{V} = H_h(\operatorname{div}) \times H_h(\operatorname{grad})^3 \times H_h(\operatorname{curl}), \\
&\Pi_{\operatorname{div}} = \left[ I,\ \Pi_h^{\operatorname{div}},\ \operatorname{curl} \right], \quad
\undertilde{B}_{\operatorname{div}} =
\begin{bmatrix}
D_{\operatorname{div}}^{-1} & 0 & 0 \\
0 & A_{\operatorname{grad}^3}^{-1} & 0 \\
0 & 0 & A_{\operatorname{curl}}^{-1}
\end{bmatrix}.
\end{aligned}
\end{equation*}

For the case $\epsilon = 0$, similar to~\eqref{HX_curl_singular}, we define
\begin{equation}
\label{HX_div_singular}
B_{\operatorname{div}} =
D_{\operatorname{div}}^{-1} +
\Pi_h^{\operatorname{div}} A_{\operatorname{grad}^3}^{-1} (\Pi_h^{\operatorname{div}})^t,
\end{equation}
which satisfies the following structure of the auxiliary space preconditioner~\eqref{iterative_aux}:
\begin{equation*}
\begin{aligned}
&B_{\operatorname{div}} = \Pi_{\operatorname{div}} \undertilde{B}_{\operatorname{div}} \Pi_{\operatorname{div}}^t, \quad \text{with} \\
&V = H_h (\operatorname{div}), \quad
\undertilde{V} = H_h(\operatorname{div}) \times H_h(\operatorname{grad})^3, \\
&\Pi_{\operatorname{div}} = \left[ I,\ \Pi_h^{\operatorname{div}} \right], \quad
\undertilde{B}_{\operatorname{div}} =
\begin{bmatrix}
D_{\operatorname{div}}^{-1} & 0 \\
0 & A_{\operatorname{grad}^3}^{-1}
\end{bmatrix}.
\end{aligned}
\end{equation*}

Therefore, the condition number \( \kappa(B_{\operatorname{div}} A_{\operatorname{div}}) \) can be estimated using \cref{Thm:condition_number_aux,Thm:condition_number_aux_singular}, as presented in \cref{Thm:HX_div}~(see also~\cite[Theorem~7.2]{HX:2007}).  
Since the proof of \cref{Thm:HX_div} is nearly identical to that of \cref{Thm:HX_curl}, we omit it here.

\begin{theorem}
\label{Thm:HX_div}
The Hiptmair--Xu preconditioner $B_{\operatorname{div}}$ defined in~\eqref{HX_div} and~\eqref{HX_div_singular} satisfies
\begin{equation*}
\kappa (B_{\operatorname{div}} A_{\operatorname{div}}) \lesssim 1.
\end{equation*}
\end{theorem}

\section{Application III: Auxiliary grid methods}
\label{Sec:Grid}
Auxiliary grid methods, first proposed in~\cite{Xu:1996} and later studied in, e.g.,~\cite{GWX:2016,WHCX:2013,ZX:2014}, provide a means of designing optimal or nearly optimal iterative methods for problems posed on unstructured grids.  
The basic idea behind auxiliary grid methods is to introduce a structured auxiliary grid that possesses a natural nested multilevel hierarchy, enabling the application of optimal multilevel preconditioners~\cite{BPX:1990,Zhang:1992}.  
In this section, we present an interpretation of auxiliary grid methods within the framework of the auxiliary space theory developed in this paper.

\subsection{Structured auxiliary grid}
Let $\Omega \subset \mathbb{R}^2$ be a bounded polygonal domain, and let $\mathcal{T}_h$ be a quasi-uniform triangulation of $\Omega$, where $h$ denotes the characteristic mesh size.  
Let $V$ denote the continuous piecewise linear finite element space defined on $\mathcal{T}_h$ with the homogeneous Dirichlet boundary condition.  
As a model problem, we consider the linear system~\eqref{model}, where the system matrix $A$ is the stiffness matrix arising from the finite element discretization of the Poisson problem using the space $V$; see, e.g.,~\cite{BS:2008} for details.

Given a uniform square partition of the entire space $\mathbb{R}^2$ with mesh size $h_0 \eqsim h$, let $\mathcal{T}_0$ denote the collection of squares that are fully contained in $\Omega$, and let $\Omega_0$ be the corresponding mesh domain determined by $\mathcal{T}_0$.  
Let $V_0$ be the continuous piecewise bilinear finite element space defined on $\mathcal{T}_0$, and let $A_0$ denote the associated stiffness matrix.

In~\cite[Lemmas~4.1 and~4.2]{Xu:1996}, the following estimates for the standard nodal interpolation operators $\Pi_h \colon V_0 \to V$ and $\Pi_0 \colon V \to V_0$ were established:
\begin{subequations}
\label{aux_grid_interpolation}
\begin{align}
\label{aux_grid_interpolation_h}
\| v_0 - \Pi_h v_0 \|_{L^2} &\lesssim h \| v_0 \|_{H^1}, &
\| \Pi_h v_0 \|_{H^1} &\lesssim \| v_0 \|_{H^1}, \quad
v_0 \in V_0, \\
\label{aux_grid_interpolation_0}
\| v - \Pi_0 v \|_{L^2} &\lesssim h \| v \|_{H^1}, &
\| \Pi_0 v \|_{H^1} &\lesssim \| v \|_{H^1}, \quad
v \in V.
\end{align}
\end{subequations}

\subsection{Auxiliary grid preconditioner}
The auxiliary grid preconditioner $B$ for the linear operator $A$ is defined by
\begin{equation}
\label{aux_grid_preconditioner}
B = D^{-1} + \Pi_h B_0 \Pi_h^t,
\end{equation}
where $D^{-1}$ is the nodal Jacobi smoother and $B_0$ is an appropriate preconditioner for $A_0$.
We observe that the structure of $B$ is consistent with the auxiliary space preconditioner given in~\eqref{iterative_aux}:
\begin{equation}
\label{aux_grid_preconditioner_aux}
B = \Pi \undertilde{B} \Pi^t, \quad \text{with} \quad
\undertilde{V} = V \times V_0,\
\Pi = [I,\ \Pi_h],\
\undertilde{B} = \begin{bmatrix} D^{-1} & 0 \\ 0 & B_0 \end{bmatrix}.
\end{equation}
Hence, the conditioner number $\kappa (BA)$ can be estimated using \cref{Thm:condition_number_aux}~(cf.~\cite{Xu:1996}).

\begin{theorem}
\label{Thm:aux_grid}
The auxiliary grid preconditioner $B$ defined in~\eqref{aux_grid_preconditioner} satisfies
\begin{equation*}
\kappa (BA) \lesssim \kappa (B_0 A_0).
\end{equation*}
\end{theorem}
\begin{proof}
Note that~\eqref{aux_grid_preconditioner_aux} implies
\begin{equation*}
\| \undertilde{v} \|_{\undertilde{B}^{-1}}^2 = \| \tilde{v} \|_D^2 + \| v_0 \|_{B_0^{-1}}^2,
\quad \undertilde{v} = ( \tilde{v}, v_0) \in \undertilde{V}.
\end{equation*}
Thanks to \cref{Thm:condition_number_aux}, the desired result follows if we verify the following:
\begin{enumerate}[(a)]
\item For any $\undertilde{v} \in V$ and $v_0 \in V_0$, we have
\begin{equation*}
\| \tilde{v} + \Pi_h v_0 \|_A^2 \lesssim \lambda_{\max} (B_0 A_0) ( \| \tilde{v} \|_{D}^2 + \| v_0 \|_{B_0^{-1}}^2 ).
\end{equation*}
\item For any $v \in V$, there exist $\tilde{v} \in V$ and $v_0 \in V_0$ such that $v = \tilde{v} + \Pi_h v_0$ and
\begin{equation*}
\| \tilde{v} \|_D^2 + \| v_0 \|_{B_0^{-1}}^2 \lesssim \lambda_{\min} (B_0 A_0)^{-1} \| v \|_A^2.
\end{equation*}
\end{enumerate}

To prove~(a), it suffices to apply a coloring argument analogous to~\eqref{Thm1:HX_curl}, together with the interpolation estimate~\eqref{aux_grid_interpolation_h}.  
Namely, we have
\begin{equation*}
\| \tilde{v} + \Pi_h v_0 \|_A^2 
\lesssim \| \tilde{v} \|_A^2 + \| \Pi_h v_0 \|_A^2
\lesssim \| \tilde{v} \|_D^2 + \| v_0 \|_{A_0}^2
\leq \| \tilde{v} \|_D^2 + \lambda_{\max}(B_0 A_0) \| v_0 \|_{B_0^{-1}}^2.
\end{equation*}

Next, we prove~(b).  
Given $v \in V$, let $v_0 = \Pi_0 v$ and $\tilde{v} = v - \Pi_h \Pi_0 v$.  
Then we get
\begin{equation}
\label{Thm1:aux_grid}
\| \tilde{v} \|_D^2 + \| v_0 \|_{B_0^{-1}}^2 = \| v - \Pi_h \Pi_0 v \|_D^2 + \| \Pi_0 v \|_{B_0^{-1}}^2.
\end{equation}
By the inverse inequality and the finite overlap property, we have
\begin{equation}
\label{Thm2:aux_grid}
\| v - \Pi_h \Pi_0 v \|_D^2 \lesssim h^{-2} \| v - \Pi_h \Pi_0 v \|_{L^2}^2.
\end{equation}
Moreover, using the interpolation estimates~\eqref{aux_grid_interpolation}, we obtain
\begin{equation}
\label{Thm3:aux_grid}
\begin{split}
\| v - \Pi_h \Pi_0 v \|_{L^2}^2
&\lesssim \| v - \Pi_0 v \|_{L^2}^2 + \| \Pi_0 v - \Pi_h \Pi_0 v \|_{L^2}^2 \\
&\lesssim h^2 \| v \|_A^2 + h^2 \| \Pi_0 v \|_{A_0}^2
\lesssim h^2 \| v \|_A^2.
\end{split}
\end{equation}
On the other hand, by~\eqref{aux_grid_interpolation_0}, we have
\begin{equation}
\label{Thm4:aux_grid}
\| \Pi_0 v \|_{B_0^{-1}}^2 
\leq \lambda_{\min}(B_0 A_0)^{-1} \| \Pi_0 v \|_{A_0}^2
\lesssim \lambda_{\min}(B_0 A_0)^{-1} \| v \|_A^2.
\end{equation}
Combining~\eqref{Thm1:aux_grid}--\eqref{Thm4:aux_grid} completes the proof of~(b).
\end{proof}

\cref{Thm:aux_grid} implies that the $A$-problem~\eqref{model} posed on the unstructured grid can be preconditioned as effectively as the corresponding $A_0$-problem defined on the structured auxiliary grid.  
Therefore, we can apply optimal preconditioners originally developed for structured grids~\cite{BPX:1990,LX:2016,Zhang:1992}.

\section{Conclusion}
\label{Sec:Conclusion}
In this paper, we presented the auxiliary space theory, which offers a unified framework for understanding a wide range of advanced techniques in numerical analysis, including subspace correction methods, Hiptmair--Xu preconditioners, and auxiliary grid methods.

A key point we emphasize is that the auxiliary space theory cleanly separates the analysis into two parts: a purely linear algebraic component and a problem-dependent component.
Since the linear algebraic aspects are handled uniformly by the auxiliary space framework, the applications discussed in \cref{Sec:MSC,Sec:HX,Sec:Grid} require only minimal additional effort in terms of algebraic manipulation.
Specifically, the remaining work involves verifying certain problem-specific properties, such as regular decompositions and interpolation estimates, along with some elementary calculations.
We highlight that this separation not only simplifies the analysis but also provides better conceptual understanding.

\bibliographystyle{siamplain}
\bibliography{refs_aux}

\end{document}